\documentclass[a4paper, 10pt]{article}
\usepackage{amssymb, latexsym, amsmath, amsthm, color, marginnote}
\usepackage{a4wide}
\usepackage{enumitem} 

\newtheorem{theorem}{Theorem}[section]

\newtheorem{lemma}[theorem]{Lemma}
\newtheorem{proposition}[theorem]{Proposition}
\newtheorem{corollary}[theorem]{Corollary}
\newtheorem{question}[theorem]{Question}

\theoremstyle{definition}

\newtheorem{remark}[theorem]{Remark}
\newtheorem{definition}[theorem]{Definition}

\newcommand{\C}{\mathbb{C}}
\newcommand{\Z}{\mathbb{Z}}
\newcommand{\N}{\mathbb{N}}

\newcommand{\link}{\rightsquigarrow}

\begin{document}

\title{Simple modules and their essential extensions for skew polynomial rings}
\author{Ken Brown, Paula A.A.B. Carvalho and Jerzy Matczuk}
\date{\today}

\maketitle

\begin{abstract} Let $R$ be a commutative Noetherian ring and $\alpha$ an automorphism of $R$. This paper addresses the question: when does the skew polynomial ring $S = R[\theta; \alpha]$ satisfy the property $(\diamond)$, that for every simple $S$-module $V$ the injective hull $E_S(V)$ of $V$ has all its finitely generated submodules Artinian. The question is largely reduced to the special case where $S$ is primitive, for which necessary and sufficient conditions are found, which however do not between them cover all possibilities. Nevertheless a complete characterisation is found when $R$ is an affine algebra over a field $k$ and $\alpha$ is a $k$-algebra automorphism - in this case $(\diamond)$ holds if and only if all simple $S$-modules are finite dimensional over $k$. This leads to a discussion, involving close study of some families of examples, of when this latter condition holds for affine $k$-algebras $S = R[\theta;\alpha]$. The paper ends with a number of open questions.

\end{abstract}

\medskip



\section{Introduction}

\subsection*{}\label{intro1} A Noetherian ring $S$ whose simple modules have the property that their finitely generated essential extensions are Artinian is said to satisfy property $(\diamond)$. For commutative rings $S$ the validity of $(\diamond)$ is due to Matlis, proved in his famous 1958 paper \cite{Matlis}; a brief survey of work on this topic in the years since then is given below, in $\S$\ref{history}. This paper focusses on $(\diamond)$ for the skew polynomial rings $S = R[\theta; \alpha]$, where $R$ is a commutative Noetherian ring and $\alpha$ is an automorphism of $R$, with the indeterminate $\theta$ satisfying the relations $\theta r = \alpha (r)\theta$ for all $r \in R$. When such a skew polynomial ring $S$ satisfies $(\diamond)$ turns out to be a surprisingly subtle question, which we do not completely settle here, and which leads naturally to other fundamental representation-theoretic questions concerning these rings.

We are led to pay particular attention to two separate but overlapping cases - first, when $S$ is a primitive ring; and second, when $R$ is an affine algebra over a field. We outline our main results for these two settings in $\S$\ref{primitintro} and $\S$\ref{affintro} respectively. Here and throughout,  given a ring $R$ and an $R$-module $V$, $E_R(V)$ will denote the $R$-injective hull of $V$.

\subsection{When $S$ is primitive}\label{primitintro} Using relatively standard methods involving the second layer condition we show that, \emph{for every commutative Noetherian ring $R$, $S = R[\theta; \alpha]$ satisfies $(\diamond)$ if and only if  $E_S(V)$ is locally Artinian  for every simple $S$-module $V$ whose annihilator $Q$ is induced from $R$.} Here, ``induced from $R$'' means that  $Q = (Q \cap R)S$. This reduction is achieved in Corollary \ref{primitive}, thereby focussing attention on the case where $S = R[\theta; \alpha]$ is primitive, since $S/(Q\cap R)S \cong \overline{R}[\theta; \overline{\alpha}]$, where $\overline{R} = R/Q\cap R$ and $\overline{\alpha}$ denotes the automorphism of $\overline{R}$ induced by $\alpha$.

Making heavy use of the characterisation of primitive skew polynomial rings \cite{JurekLeroy}, we prove:

\begin{theorem}\label{first}  Let $R$ be a commutative Noetherian ring, let $\alpha$ be an automorphism of $R$, and let $S = R[\theta;\alpha]$. Suppose that $S$ is primitive.
\begin{enumerate}
\item If $R$ has Krull dimension 0 then $S$ satisfies $(\diamond)$.
\item Suppose that $R$ contains an uncountable field. Suppose also that either $R$ has Krull dimension at least 2, or $\mathrm{Spec}(R)$ is uncountable. Then $S$ does not satisfy $(\diamond)$.
\end{enumerate}
\end{theorem}

The above is Theorem \ref{primitivecase}; its proof occupies $\S\S$4, 5. Clearly, the above necessary and sufficient conditions don't exhaust all possibilities, and we leave the closure of the gap between (a) and (b) as one of a number of open questions raised by the paper.

\subsection{When $R$ is affine}\label{affintro} Theorem \ref{first} is however sufficient to settle the case where $R$ is a finitely generated algebra over an uncountable field $k$, with $\alpha$ a $k$-algebra automorphism:

\begin{theorem}\label{affine1} Let $k$ be an uncountable field and $R$  an affine $k$-algebra, and let $\alpha$ be a $k$-algebra automorphism of $R$. Let $S = R[\theta ; \alpha]$. Then the following are equivalent:
\begin{enumerate}
 \item $S$ satisfies $(\diamond)$;
 \item all simple $S$-modules are finite dimensional $k$-vector spaces.
\end{enumerate}
\end{theorem}

This is Theorem \ref{affine}. The direction $(b)\Rightarrow(a)$ follows from known considerations based on the second layer condition, and doesn't require the cardinality hypothesis on $k$. There is thus a further obvious question: is $(a) \Rightarrow (b)$ true when $k$ is countable? Although we don't know the answer to this question in general, we can at least answer it in the most obvious special case, namely when $R$ is a polynomial algebra and $\alpha$ is a linear automorphism:

\begin{theorem}\label{affine2} Let $k$ be a field, $t$ a positive integer, $V$ a vector space over $k$ with basis $\{x_1, \ldots , x_t\}$ and $\alpha \in \mathrm{GL}(t,k)$ an automorphism of $V$. Let $R = k[x_1, \ldots , x_t]$, so that $\alpha$ induces a $k$-algebra automorphism of $R$, also denoted by $\alpha$. Then $S := R[\theta; \alpha]$ satisfies $(\diamond)$ if and only if $|\alpha|< \infty$.
\end{theorem}

The finiteness of $|\alpha|$ featuring in the above result is known to be equivalent, for an \emph{arbitrary} semiprime commutative coefficient ring $R$, to $S = R[\theta; \alpha]$ satisfying a polynomial identity (PI), \cite{PascaudValette}. This might lead us to suspect that a \emph{third} equivalent condition, ``\emph{$S$ satisfies a PI}'', could be added to Theorem \ref{affine1} also. This is not the case, however, as we show by Example \ref{groupring}.

The existence of this and other examples provokes the obvious question: for which $k$-affine Noetherian skew polynomial algebras $S = R[\theta; \alpha]$ are all the simple $S$-modules finite dimensional $k$-spaces? Our contribution to this question is chiefly to demonstrate its complexity, by means of detailed analysis of some families of examples in $\S$\ref{examples}. Most notably, we undertake a detailed study of the algebras $S_{\C, t, \alpha} = \C[x_1, \ldots , x_t][\theta; \alpha]$ when $t =1$ or $2$ and $\alpha$ is an \emph{arbitrary} $\C$-algebra automorphism. But even for $t = 2$ our analysis is incomplete, and leads to delicate issues having connections to dynamical systems and to algebraic geometry.

\subsection{Historical background}\label{history}
For commutative Noetherian rings $(\diamond)$ is an immediate consequence of the Artin-Rees property, formally recorded as part of Matlis's seminal 1958 paper \cite{Matlis} on injective modules over such rings. In 1959  Philip Hall proved $(\diamond)$ for group rings $RG$ of finitely generated nilpotent groups $G$, provided $R$ is either $\Z$ or a locally finite field. In 1974 this result was extended, independently by Jategaonkar \cite{Jategaonkar74a} and Roseblade \cite{Roseblade}, to polycyclic-by-finite groups $G$, building on earlier celebrated work of Roseblade \cite{Roseblade1} on the finite dimensionality of the simple $RG$-modules for these group rings. Hall and Roseblade were motivated by applications to the structure of finitely generated soluble groups.

Motivated by applications to Jacobson's conjecture, Jategaonkar proved in 1974 \cite{Jategaonkar74} that $(\diamond)$ is satisfied by fully bounded Noetherian rings, thus incorporating Noetherian rings satisfying a polynomial identity (PI) and so generalising the commutative case. Musson \cite{Musson80} gave the first examples of Noetherian rings for which $(\diamond)$ fails, by showing that the group algebra $kG$ of a polycyclic-by-finite group $G$ over a field $k$ which is \emph{not} locally finite satisfies $(\diamond)$ only if $G$ is abelian-by-finite, that is only if $kG$ satisfies a PI. In so doing he thus delineated the limits of the earlier results of Hall and Roseblade.

More recent work has discussed $(\diamond)$ for differential operator rings \cite{CarvalhoHatipogluLomp}, \cite{SV}, down-up algebras \cite{CarvalhoLompPusat} and quantised Weyl algebras \cite{CarvalhoMusson}. \cite{Musson12} gives a brief survey of results on property $(\diamond)$ up to 2010.

\subsection{Layout}\label{lay} Preliminary observations and notation regarding property $(\diamond)$ and skew polynomial rings are in $\S 2$. $\S 3$ contains a summary of the necessary background on the second layer condition, leading up to Corollary \ref{primitive}, which essentially allows us to focus on the case where $S = R[\theta; \alpha]$ is a primitive ring. The analysis of primitive skew polynomial rings is contained in $\S\S4$ and $5$: the key result of $\S4 $ is the construction of a faithful simple $S$-module whose injective hull is not locally Artinian when $R$ is an $\alpha$-simple domain which is not a field. This is Proposition \ref{simplenotdiamond1}, which then allows us to deduce Theorem \ref{first} (= Theorem \ref{primitivecase}). In $\S 6$ Theorem \ref{first} is applied in the setting where $R$ is an affine algebra over the field $k$ and $\alpha$ is an algebra automorphism, to deduce Theorems \ref{affine1} and \ref{affine2} (= Theorems \ref{affine} and \ref{poly}). $\S 7$ is devoted to a careful analysis of the simple modules and the prime spectra of a number of examples, and families of examples, of skew polynomial algebras over a commutative affine Noetherian domain $R$. These examples may well have interest beyond the immediate question at hand, namely the validity of $(\diamond)$. Finally, in $\S 8$ we gather together and briefly discuss the open questions which have arisen in the course of this work.

All rings considered are associative with identity and all modules are unitary and are right modules unless stated otherwise.

\section{Preliminaries}

\subsection{Formulation and preservation of $(\diamond)$}\label{dimondform} Recall that a module $M$ is a subdirect product of a family of modules $\{F_{\lambda}\}$ if there exists an embedding $\iota:M\rightarrow \prod_{\lambda\in\Lambda} F_{\lambda}$ into the product of the modules $F_{\lambda}$ such that, for each projection $\pi_{\mu}:\prod_{\lambda\in\Lambda} F_{\lambda}\rightarrow F_{\mu}$, the composition $\pi_{\mu}\iota$ is surjective. The following lemma is due to Hatipo\"glu and Lomp, \cite{ChristianCan}.

\begin{lemma}\label{Christian}$\textrm{\cite[Lemma 2.1]{ChristianCan}}$  Given a ring $R$, the following conditions are equivalent:
\begin{enumerate}
 \item $R$ satisfies $(\diamond)$;
 \item every right $R$-module is a subdirect product of locally Artinian modules;
 \item every finitely generated right $R$-module is a subdirect product of Artinian modules.
\end{enumerate}
\end{lemma}

 Recall that a ring extension $T\subseteq S$ is called a \emph{finite normalizing extension} if there exists a finite set $\{s_1,\ldots, s_n\}$ of elements of $S$ such that $S=\sum_{i=0}^{n} s_iT$, with $s_iT=Ts_i$, for $i = 1, \ldots , n$. Part (a) of the next result is due to Hatipo\"glu and Lomp, \cite[Proposition 2.2]{ChristianCan}, while (b) is adapted from the work of Hirano  \cite[Theorems 1.8, 1.11]{Hirano00}.

\begin{proposition}\label{reduction}
Let $S$ be a finite normalizing extension of a ring $T$.
\begin{enumerate}
 \item If $T$ satisfies $(\diamond)$, then so does $S$.
 \item Assume also that $T$ is Noetherian and a direct summand of $S$ as a left $T$-module. If $S$ satisfies $(\diamond)$, then so does $T$.
 \item If $I$ is an ideal of $S$ and $S$ satisfies $(\diamond)$, then so does $S/I$.
\end{enumerate}
\end{proposition}
\begin{proof}
(a) See \cite[Proposition 2.2]{ChristianCan}.

(b) Let $M$ be a finitely generated right $T$-module, so $M\otimes_TS$ is a finitely generated right $S$-module. By hypothesis and Lemma \ref{Christian} there exists a family $\{M_{\lambda}\}$ of $S$-submodules of $M\otimes_T S$ such that each  $(M\otimes_T S)/M_{\lambda}$ is Artinian and $\bigcap_{\lambda}M_{\lambda}=0$. Since $T$ is a direct summand of $S$ as a left $T$-module, $M$ can be identified with a right $T$-submodule of $M\otimes_T S$. Fix $\lambda$. Then $M/(M\cap M_{\lambda})$ is isomorphic to a  $T$-submodule of $(M\otimes_T S)/M_{\lambda}$. By \cite[Theorem 4]{FormanekJategaonkar}, $(M\otimes_T S)/M_{\lambda}$ is Artinian as an $S$-module if and only if it is Artinian as a $T$-module. Now, note that $\bigcap_{\lambda}(M\cap M_{\lambda})=0$, so the result follows from Lemma \ref{Christian}(c).

(c) This is trivial, since $E_{S/I}(V) \subseteq E_S(V)$ for any $S/I$-module $V$.
\end{proof}

\subsection{Skew polynomial algebras of automorphism type}\label{skewpoly}

 Let $R$ be a ring and let $\alpha$ be an automorphism of $R$. The \emph{skew polynomial ring of automorphism type}, $S := R[\theta; \alpha]$, is the ring of polynomials in $\theta$ with coefficients in $R$ subject to the relation $\theta r=\alpha(r)\theta$ for all $r \in R$.

 \begin{equation*} \textit{We shall maintain the notation } S, R, \theta, \alpha \textit{ throughout the paper.}
  \end{equation*}

 Often, $R$ will in addition be commutative or Noetherian, but we will nevertheless state these hypotheses as required in our results, for emphasis. A right ideal $I$ of $R$ is said to be \emph{$\alpha$-stable} if $\alpha(I)=I$. We say that $R$ is \emph{$\alpha$-simple} if $(0)$ and $R$ are the only $\alpha$-stable ideals of $R$. An $\alpha$-stable ideal $P$ of $R$ is \emph{$\alpha$-prime} if, for all $\alpha$-stable ideals $I$ and $J$ of $R$, $IJ\subseteq P$ implies that $I\subseteq P$ or $J\subseteq P$. The ring $R$ is said to be \emph{$\alpha$-prime} if its ideal $(0)$ is $\alpha$-prime, and is called $\alpha$-\emph{simple} if its only $\alpha$-stable ideals are $(0)$ and $R$.

The following well known facts can be found in \cite[$\S$10.6]{MR}.

\begin{lemma}\label{stable}
Let $R$ be a right Noetherian ring, $\alpha$ an automorphism of $R$ and $S=R[\theta;\alpha]$.
\begin{enumerate}
 \item A right ideal $I$ of $R$ is $\alpha$-stable if and only if $\alpha(I)\subseteq I$.
 \item If $I$ is an $\alpha$-stable ideal of $R$, then $IS$ is an ideal of $S$ and the ring $S/IS$ is isomorphic to $(R/I)[\theta; {\overline{\alpha}}]$, where $\overline{\alpha}$ is the automorphism of $R/I$ induced by $\alpha$. The ideal $IS$ is prime if and only if $I$ is $\alpha$-prime.
 \item If $P$ is a prime ideal of $S$ such that $\theta\notin P$ then $P\cap R$ is an $\alpha$-prime ideal of $R$.
 \item If $P$ is an $\alpha$-prime ideal of $R$, then there exists $n\in\N$ and a minimal prime $Q$ over $P$ such that $ P=\bigcap_{i=0}^n \alpha^i(Q)$.
\end{enumerate}
\end{lemma}

A convenient mechanism to construct interesting simple $S$-modules is as follows.
\begin{lemma}\label{artinian}
 Suppose that  $u$ is a central unit of a right Noetherian ring $R$ and $S=R[\theta;\alpha]$.
 \begin{enumerate}
   \item The lattice of right $\alpha$-ideals of $R$ is isomorphic to the lattice of $S$-submodules of $S/(u-\theta)S$.
  \end{enumerate}
Suppose additionally in (b) and (c) that $R$ is commutative.
\begin{enumerate}[resume]
   \item The right $S$-module $S/(u-\theta)S$ is   simple if and only if $R$  is $\alpha$-simple.
   \item  Suppose that $R$ is also a domain with no proper idempotent ideals; for example, $R$ could be a Noetherian domain. Then  $ S/(u-\theta)S$ is an Artinian   right $S$-module  if and only if  $R$ is $\alpha$-simple.
 \end{enumerate}
\end{lemma}
\begin{proof}
(a) Notice that, for any element $r\in R$,
\begin{equation}\label{calc} r\theta=\theta\alpha^{-1}(r)=-(u-\theta)\alpha^{-1}(r)+u\alpha^{-1}(r).
\end{equation}
Let $N$ be an $\alpha$-stable right ideal of $R$. Then (\ref{calc}) and the centrality of  $u$ in $R$ yield  that $J :=(u-\theta)S+N$ is a right ideal of $S$.

Let $J$ be a right ideal of $S$ with $(u-\theta)S \subseteq J$. Let $w\in J$. Using division by the monic polynomial $u-\theta$ we can find $h\in S$ and $r\in R$ such that $w=(u-\theta)h+r$
and $J=(u-\theta)S+N$ follows for $N :=J\cap R$. Moreover, applying (\ref{calc}) with $r \in N$ and using our assumptions on $u$, it follows that $\alpha^{-1}(N)\subseteq N$. By Lemma \ref{stable}(a), $N$ is an $\alpha$-stable right ideal of $R$. This completes the proof of (a).

(b) This is a direct consequence of (a).

(c) Suppose that $S/(u-\theta)S$ is Artinian as a right $S$-module. Let $I$ be a proper $\alpha$-stable ideal  of $R$.   It follows from (a) that there exists $n\in \N$ such that $I^n=I^{2n}$. The stated hypotheses force $I$ to be $(0)$, so $R$ is $\alpha$-simple.  The converse is given by (b). The Krull Intersection Theorem \cite[Corollary 5.4]{E} ensures that Noetherian domains satisfy the stated hypothesis.
  \end{proof}

\section{Reduction to the primitive case}

\subsection{The second layer condition}\label{slc}

Given a non-zero module $M$ over a right Noetherian ring $T$,  an \emph{affiliated submodule} of $M$ is a submodule of the form $\mathrm{Ann}_M(P)=\{m\in M:mP=0\}$, where $P$ is an ideal of $T$ which is maximal amongst the annihilators of non-zero submodules of $M$. It is easy to see that such an ideal $P$ is a prime ideal of $T$, \cite[Proposition 3.12]{GoodearlWarfield}. An \emph{affiliated series} for $M$ is a series $0=M_0 \subset M_1 \subset \cdots \subset M_n=M$ of submodules of $M$, such that for each $i\in\{1,\ldots,n\}$, $M_i/M_{i-1}$ is an affiliated submodule of $M/M_{i-1}$. The ideals $P_i:=\mathrm{Ann}_T(M_i/M_{i-1})$ are called the \emph{affiliated primes} of $M$ with respect to the given series. Full details are in \cite[Chapter 8]{GoodearlWarfield}, for example.

We briefly recall here the key ideas we need from the theory of prime links for Noetherian rings. For more details, see for example \cite[Chapter 12]{GoodearlWarfield}, \cite{Jategaonkar86}. Let $T$ be a Noetherian ring and $M$ a finitely generated right $T$-module. Suppose that
\begin{equation}\label{extend} 0\subset U \subset M
  \end{equation}
is an affiliated series of $M$, and suppose that $U$ is an essential submodule of $M$, with corresponding affiliated prime ideals $Q$ and $P$, so that $UQ = 0 = (M/U)P$. To understand the possible relations between the modules $U$ and $M/U$, normalise (\ref{extend}) by replacing $M$ if necessary with a submodule $M'$ of $M$ properly containing $U$ such that $I:=\mathrm{Ann}_T(M')$ is maximal amongst the annihilators of those submodules of $M$ which properly contain $U$. Notice that $PQ \subseteq I$. With this normalisation, we continue to write $M$ for the replacement module.

Then, by Jategaonkar's so-called Main Lemma, \cite[Theorem 12.1]{GoodearlWarfield}, \cite{Jategaonkar86}, there are two possibilities: either

\indent (a) $PQ\subseteq I\subset P\cap Q$, and $(P\cap Q)/I$ is torsion-free as a left $T/P$-module and as a right $T/Q$-module; or

\indent (b) $P\subset Q$ and $MP=0$.

 In partial converse, if (a) holds for primes $P$ and $Q$ of $T$, then a right $T$-module $M$ exists as above, with $U$ [resp. $M/U$] being $T/Q $ [resp. $T/P$] - torsion-free, \cite[Theorem 12.2]{GoodearlWarfield}. If case (a) holds, we say that \emph{there is a link from $P$ to $Q$}, and we write $P\link Q$. If case (a) holds and $U$ is $T/Q$-torsion-free, then $M/U$ is $T/P$-torsion-free. On the other hand, if (b) holds, then both $M$ and $M/U$ are $T/P$-torsion.

 A Noetherian ring $T$ is said to satisfy the (right) \emph{strong second layer condition (s.s.l.c.)} if, for every prime ideal $Q$ of $T$, only case (a) can occur in the setting of (\ref{extend}). The formally weaker (right) \emph{second layer condition (s.l.c.)} holds for $T$ if, for all primes $Q$, only case (a) occurs when $U$ is in addition required to be $T/Q$-torsion-free.


Those Noetherian rings satisfying the s.s.l.c. form an important and large subclass. For our purposes, the key result in this direction is the following proposition. Recall (for example, from \cite[page 224]{GoodearlWarfield}) that a Noetherian ring $T$ is \emph{AR-separated} if, for every prime ideal $Q$ of $T$ and ideal $I$ with $Q \subset I \subset T$, there is an ideal $J$ of $T$ with $Q \subset J \subseteq I$ such that $J/Q$ has the Artin-Rees property in $T/Q$.

\begin{proposition}\label{ARsep1}\cite[Theorem 13.4]{GoodearlWarfield} If the Noetherian ring $T$ is AR-separated, then it satisfies the s.s.l.c.
\end{proposition}

For example, Noetherian rings satisfying a polynomial identity and enveloping algebras of finite dimensional solvable Lie algebras are AR-separated, \cite{Jategaonkar86}. Of more relevance for us, however, is:

\begin{proposition}\label{ARsep}\cite{Poole90}
Let $R$ be a commutative Noetherian ring, $\alpha\in Aut(R)$ and $S=R[\theta; \alpha]$.
\begin{enumerate}
 \item Then $S$ is AR-separated, and hence satisfies s.s.l.c.
 \item Let $P$ be a prime ideal of $S$ such that $(P\cap R)S=P$. Then $P$ has the Artin-Rees property. In particular, if $P\link Q$ or $Q\link P $, then $Q = P$.
\end{enumerate}
\end{proposition}

\subsection{The second layer condition and $(\diamond)$ } \label{diamondslc}

Lenagan's Lemma, \cite[Theorem 7.11]{GoodearlWarfield}, guarantees that if $T$ is a Noetherian ring and $P$ and $Q$ are prime ideals of $T$ with $P\link Q$, then $T/P$ is Artinian if and only if $T/Q$ is Artinian. From this and Jategaonkar's Main Lemma the following well-known consequence follows easily:

\begin{proposition}\label{sslc}
Let $T$ be a Noetherian ring satisfying the s.l.c. and let $V$ be a simple right $T$-module with $Q := \mathrm{Ann}_T(V)$. Suppose that $T/Q$ is Artinian. Then every finitely generated essential extension of $V$ is Artinian.
\end{proposition}

We are now in a position to deal with property $(\diamond)$ for many simple modules over skew polynomial rings.

\begin{theorem}\label{goodprimes}
Let $R$ be a commutative Noetherian ring, $\alpha\in Aut(R)$ and $S=R[\theta;\alpha]$. Let $V$ be a simple right $S$-module and let $Q=ann_S(V)$. If $(Q\cap R)S\subset Q$, then $S/Q$ is Artinian and (hence) every finitely generated essential extension of $V$ is Artinian.
\end{theorem}

\begin{proof} Suppose first that $\theta\in Q$. Then $Q/\theta S$ is a primitive ideal of the commutative ring $S/\theta S\simeq R$, so $S/Q$ is a field. Therefore the desired property of $V$ follows from Propositions \ref{ARsep}(a) and \ref{sslc}.

Suppose now that $\theta\notin Q$ and $(Q\cap R)S\subset Q$. Then $Q\cap R$ is $\alpha$-prime by Lemma \ref{stable}(c). By \cite[Theorem 4.3]{Irving79}
the order of the automorphism of $R/(Q\cap R)$ induced by $\alpha$ is finite. Hence, by \cite[Corollary 10]{DamianoShapiro}, $S/(Q\cap R)S\simeq (R/(Q\cap R))[\theta; \alpha|_{R/(Q\cap R)}]$ is a ring satisfying a polynomial identity. By Kaplansky's Theorem \cite[Theorem I.13.3]{BrownGoodearl} applied to the primitive ideal $Q/(Q\cap R)S$ of $S/(Q\cap R)S$, we deduce once again that $S/Q$ is Artinian. Thus, again, $E_S(V)$ is locally Artinian by Propositions \ref{ARsep}(a) and \ref{sslc}.
\end{proof}

The following corollary of Theorem \ref{goodprimes} in large part reduces the analysis of $(\diamond)$, for skew polynomial rings $S$ of automorphism type, to the case where $S$ is a primitive ring. Recall that a Noetherian ring $T$ is \emph{polynormal} if for every
pair of distinct ideals $I$ and $J$ of $T$ with $I\subset J$, there is an element $a\in J\setminus I$ such
that $aS + I= Sa + I$.

\begin{corollary}\label{primitive} Let $R$ be a commutative Noetherian ring, $\alpha\in Aut(R)$ and $S=R[\theta; \alpha]$. Consider the following statements:
\begin{enumerate}
\item $S$ satisfies $(\diamond)$;
\item $E_S(V)$ is locally Artinian for every simple right $S$-module $V$ whose annihilator $Q$ satisfies $Q=(Q\cap R)S$;
\item every primitive factor of $S$ of the form $S/(P\cap R)S$ satisfies $(\diamond)$.
\end{enumerate}
Then (a) and (b) are equivalent and imply (c). If additionally either $S$ is polynormal or every primitive ideal $P$ of $S$ of the form $P=(P\cap R)S$  is generated by a normal element, then all the above statements are equivalent.
\end{corollary}
\begin{proof}
(a)$\Leftrightarrow$(b) is a consequence of Theorem \ref{goodprimes}. That (a)$\Rightarrow$(c) follows from  Proposition \ref{reduction}(c).

If $S$ satisfies one of the additional hypotheses then (c)$\Rightarrow$(a) follows from Theorem \ref{goodprimes} combined with \cite[Theorem 9.3.4]{Jategaonkar86}; see also \cite[Lemma 6.1 and Lemma 6.3]{KenWarfield}.
\end{proof}







\section{The primitive case}\label{prim}

\subsection{Primitive skew polynomial rings}\label{simples}

 We begin by recalling the results of \cite{JurekLeroy}, where Leroy and Matczuk presented necessary and sufficient conditions for the primitivity of $S =R[\theta; \alpha]$; see also \cite{Jordan}.

\begin{definition}\label{special} Given a ring $T$ and $\alpha\in \mathrm{Aut}(T)$, $T$ is \emph{$\alpha$-special} if there is an element $a$ of $T$ such that the following conditions are satisfied.
\begin{enumerate}
 \item For all $n\geq 1$, $N_n^{\alpha}(a):=a\alpha(a)\ldots\alpha^{n-1}(a)\neq 0$.
 \item For every non-zero $\alpha$-stable ideal $I$ of $T$, there exists $n\geq 1$ such that $N_n^{\alpha}(a)\in I$.
\end{enumerate}
When this occurs, the element $a$ is called an \emph{$\alpha$-special element}.
\end{definition}

From the definition it follows easily that an $\alpha$-special ring is $\alpha$-prime. Clearly, an $\alpha$-simple ring is $\alpha$-special, with $1$ as $\alpha$-special element in this case. However, there are $\alpha$-special rings which are not $\alpha$-simple. Consider, for example, the ring $R$ of formal power series $k[[X]]$, where $k$ is a field containing an element $q$ which is not a root of unity. Let $\alpha$ be the $k$-algebra automorphism defined by setting $\alpha (X) := qX$. Then $X$ is an $\alpha$-special element of $R$, but $\langle X \rangle$ is a proper $\alpha$-ideal of $R$.

Here is the characterisation of primitivity:

\begin{theorem}{\cite[Theorem 3.10]{JurekLeroy}}\label{JurekLeroy}
Let $R$ be a commutative Noetherian ring and let $\alpha \in \mathrm{Aut}(R)$. Then $S=R[\theta; \alpha]$ is primitive if and only if $R$ is $\alpha$-special and $\alpha$ has infinite order.
\end{theorem}

In the proof of Theorem \ref{JurekLeroy}, given a Noetherian PI ring $R$ and an automorphism $\alpha$ of $R$ with $\alpha$-special element $a$, the authors build a simple faithful module over $S=R[\theta; \alpha]$ of the form $S/M$, where $M$ is a maximal right ideal of $S$ containing $(1-a\theta)S$. Following similar ideas, we show in the proposition below that, when $R$ is commutative, $(1-a\theta)S$ is actually a maximal right ideal; this will be important for us in the sequel. Note also that the proposition provides a proof that the conditions on $\alpha$ in Theorem \ref{JurekLeroy} are sufficient for primitivity of $S$ when $R$ is commutative.

\begin{proposition}\label{simplemodule}
Let $R$ be a commutative Noetherian ring and let $\alpha\in Aut(R)$ be such that $R$ is $\alpha$-special with $a$ an $\alpha$-special element of $R$. Let $S=R[\theta;\alpha]$. Then $(1-a\theta)S$ is a maximal right ideal of $S$. If in  addition $\alpha$ has infinite order, then $V:=S/(1-a\theta)S$ is a faithful right $S$-module.
\end{proposition}

\begin{proof}
First note that since $a$ is a regular element of $R$ by \cite[Lemma 1.2]{JurekLeroy}, $(1-a\theta)S \neq S$. Let $J$ be a right ideal of $S$  such that $M:=(1-a\theta)S\subset J$.
We claim that
\begin{equation}\label{hitR} J \cap R \quad \neq \quad (0).
\end{equation}
Let $n \geq k$ and let
$$f=a_n\theta^n+\ldots+a_k\theta^k\in J\backslash (1-a\theta )S,$$
with $f$ of shortest possible length $\ell(f) :=n-k$. Suppose that $\ell (f)>0$. Then, since $1\equiv a\theta \;(\mbox{mod } M)$, $$a_k\theta^k\equiv a\theta a_k\theta^k\equiv a\alpha(a_k)\theta^{k+1} \;(\mbox{mod } M).$$
Thus, $f\equiv g \;(\mbox{mod } M),$ where
$$g:=a_n\theta^n+\ldots + (a_{k+1}+a\alpha(a_k))\theta^{k+1},$$
so that $\ell (g) < \ell(f)$. This contradicts our choice of $f$, and so $\ell (f)=0$. That is, $0 \neq a_n\theta^n\in J$. But now $J$ also contains $\alpha^{-1}(a_n)\theta^{n-1}$, because
$$ a_n\theta^n\alpha^{-n}(a)= a a_n\theta^n = a\theta \alpha^{-1}(a_n)\theta^{n-1}\equiv \alpha^{-1}(a_n)\theta^{n-1}(\mbox{mod } M).$$
This implies that $\alpha^{-n}(a_n)\in J$, so (\ref{hitR}) holds.

Next, observe that $J\cap R$ is an $\alpha$-stable ideal of $R$, since, if $r\in J\cap R$, then
$$ \alpha^{-1}(r)\equiv a\theta\alpha^{-1}(r) =r\theta \alpha^{-1}(a) \;(\mbox{mod } M),$$
so that $\alpha^{-1}(r)\in J$.  Therefore, as $a$ is an $\alpha$-special element of $R$ and $J\cap R$ is non-zero and $\alpha$-stable, there exists $k\in \N$ such that $(a\theta)^k\in J$. But then $(a\theta)^{k-1}=(1-a\theta)(a\theta)^{k-1}+(a\theta)^k$ and so $(a\theta)^{k-1}\in J$. It follows that $1\in J$ and $(1-a\theta)S$ is maximal, as required.

To prove the last statement, assume that $\alpha$ is of infinite order. Let $P=\mathrm{Ann}_S(V)$. Then $P\subseteq (1-a\theta)S$, so $P\cap R= (0)$ and $P$ does not contain $\theta$. By \cite[Theorem 4.3]{Irving79} as $R$ is $\alpha$-prime and $\alpha$ has infinite order, every non-zero prime ideal of $S$ which intersects $R$ in $(0)$ contains $\theta$, so $P=(0)$.
\end{proof}

\begin{remark}\label{torfree}
For use in Section \ref{necsuff}, we record here a further property of the simple module $V$ defined in Proposition \ref{simplemodule}. Retain the notation of the proposition, and define $\mathcal{A}$ to be the multiplicative subsemigroup of $R \setminus \{0\}$ generated by $\{\alpha^i (a) : i \geq 0 \}$. Then $\mathcal{A}$ consists of regular elements of $S$, and is easily seen to form an Ore set in that ring. We claim that
$$ V  \textit{is an } \mathcal{A}\textit{-torsion free simple right } \mathit{S} \textit{-module.} $$

By  the proposition, $V$ is simple, so it is either torsion or torsion free. Since $a$ is a regular element of $S$, for all $h \in S\setminus \{0\}$, the degree in $\theta$ of $(1-a\theta)h$ is one more than the degree of $h$. Hence $(1-a\theta)S\cap R=0$. Therefore $1+(1-a\theta)S\in V$ is not $\mathcal A$-torsion, and so $V$ is ${\mathcal A}$-torsion free.
\end{remark}

\subsection{Reduction to the case where $R$ is a domain}\label{domain}

To achieve the reduction as in the title of the subsection we need the following lemma: 

\begin{lemma}\label{npower} Let $R$ be a commutative Noetherian ring, let $\alpha\in Aut(R)$ and set $S=R[\theta;\alpha]$. Suppose that $S=R[\theta;\alpha]$ is prime. Then:
\begin{enumerate}
\item $R$ is $\alpha$-prime, and so there exist $n \in \N$ and a prime ideal $Q$ of $R$ such that

$\{Q,\alpha(Q), \ldots , \alpha^{n-1}(Q)\}$ is the set of minimal primes of $R$, with $\cap_{i=0}^{n-1}\alpha^i(Q) = (0)$.

\item In the setting of (a), if $S$ is primitive then $(R/Q)[\theta^n; \alpha^n]$ is primitive.
\end{enumerate}
\end{lemma}
\begin{proof}(a) This is clear from Lemma \ref{stable}(c),(d).

(b) With the notation of (a), let $T := R[\theta^n ; \alpha^n]$. Notice that $QT$ and its $\alpha$-conjugates are the minimal prime ideals of $T$. Suppose that $V$ is a faithful simple right  $S$-module. By \cite[Theorem 4]{FormanekJategaonkar} $V$ has a (finite) composition series as a $T$-module,
$$ 0 = V_0 \subset V_1 \subset \cdots \subset V_t = V.$$
Set $P_i := \mathrm{Ann}_T(V_i/V_{i-1})$, for $i = 1, \ldots , t$, so that
$$ V(P_t P_{t-1} \ldots P_1) = 0.$$
Since $V$ is by hypothesis a faithful $S$-module, $ P_tP_{t-1}\ldots P_1 = (0)$. Hence there exists $j$, $1 \leq j \leq t$, with $P_j \subseteq QT$, and the minimality of the prime ideal $QT$ of $T$ ensures that $P_j = QT$.

That is, $V_j/V_{j-1}$ is a faithful simple $T/QT$-module, as required.
\end{proof}

The above lemma, with an equivalence in (b), can be found as \cite[Corollary 2.2]{JurekLeroy} with a different proof.

\begin{lemma}\label{power}
Let $R$ be a commutative Noetherian ring, $\alpha$ an automorphism of $R$, and $S=R[\theta; \alpha]$.
\begin{enumerate}
\item Let $n \geq 1$ and let $T=R[\theta^n; \alpha^n]$. Then $S$ satisfies $(\diamond)$ if and only if  $T$ satisfies $(\diamond)$ .
\item Suppose that $R$ is $\alpha$-prime, with $Q$ and $n$ chosen as in Lemma \ref{npower}(a). If $S$ satisfies $(\diamond)$ so does $T/QT\simeq (R/Q)[\theta^n; \alpha^n]$.
\end{enumerate}
\end{lemma}

\begin{proof} (a) This follows immediately from Proposition \ref{reduction}, since $S$ as a $T$-module is free with a basis formed by the normal elements $1, \theta, \theta^2, \ldots, \theta^{n-1}.$

(b) Clear from (a) and Proposition \ref{reduction}(c).
\end{proof}

\subsection{$(\diamond)$ when $R$ is an $\alpha$-simple domain}

To investigate which primitive skew polynomial rings over a commutative Noetherian domain $R$ satisfy $(\diamond)$, we first consider the case when $R$ is $\alpha$-simple. We preface the proposition with three lemmas needed for its proof.

\begin{lemma} \label{orbital} Let $\alpha$ be an automorphism of the commutative ring $R$, and suppose that the $\alpha$-orbit of every non-zero prime ideal of $R$ is infinite. Let $\mathcal{P} = \{P_1, \ldots , P_t \}$ be a finite set of non-zero prime ideals of $R$. Then there exists a positive integer $n$ such that, for all $i = 1, \ldots , t$ and for all $j \in \mathbb{Z} \setminus \{0\}$,
$$ \alpha^{jn} (P_i) \notin \mathcal{P}. $$
\end{lemma}
\begin{proof} For each $i$, there is a positive integer $n_i$ such that $\{\alpha^{n_it}(P_i) : t \in \Z \} \cap \mathcal{P} = \{P_i\}$. Thus $n := \prod_i n_i$ is as required.
\end{proof}

\begin{lemma} \label{notartin} Let $\alpha$ be an automorphism of the commutative ring $R$, let $S = R[\theta; \alpha]$, and let $\rho$ be a non-zero non-unit of $R$. Then $S/\rho S$ is not an Artinian right $S$-module.
\end{lemma}
\begin{proof} We claim that the chain
$$ \theta S + \rho S \supseteq \theta^2 S + \rho S \supseteq \cdots \supseteq \theta^m S + \rho S \supseteq \cdots, $$
of right ideals of $S$ is strictly descending.
In order to prove this, suppose that $m \geq 0$ and $\theta^m \in \theta^{m+1} S + \rho S=S \theta^{m+1} + \rho S,$ say
$\theta^m = p\theta^{m+1} + \rho q$
 for $p,q \in S.$ Then $\rho q\in \rho S \theta^{m}$. This implies that $1\in S\theta + \rho S$, contradicting the fact that $\rho$ is not invertible.
\end{proof}

\begin{lemma}\label{form} Let $\alpha$ be an automorphism of the ring $R$, $S = R[\theta; \alpha]$, and let $\rho$ be a regular non-unit of $R$. Let $\{r_{\lambda} : \lambda \in \Lambda \}$ be a set of coset representatives  for $\rho R$ in $R$. Then every element $m$ of $M := S/\rho(1 - \theta )S$ has a unique expression
$$ m \; = \; \sum_{i=\ell}^k r_{\lambda_i}\theta^i + \rho b + \rho(1 - \theta)S, $$
where $b \in R$ and either (i) the part of the expression under the summation symbol is $0$, or else (ii) $\ell \leq k \in \mathbb{Z}_{\geq 0}$ with $r_{\lambda_{\ell}}, r_{\lambda_k} \neq 0$.
\end{lemma}
\begin{proof} Since $S = R \oplus (1 - \theta)S$ as right $R$-modules and $S \cong \rho S$ as right $S$-modules, the desired expression using coset representatives $\{ \rho b : b \in R \}$ for the elements of the submodule $\rho S/\rho (1 - \theta)S$ of $M$ is clear. Representation of elements of $M$ by the listed elements follows, since $S/\rho S \cong R/\rho R \otimes_R S$ as right $S$-modules.
\end{proof}

\begin{proposition}\label{simplenotdiamond1} Let $R$ be a commutative Noetherian domain which is not a field, $\alpha$ an automorphism of $R$ and suppose that $R$ is $\alpha$-simple. Then $S= R[\theta;\alpha]$ does not satisfy $(\diamond)$. Moreover, there is $n\in \N$ such that $S/(1-\theta^n)S$ is a finite length $S$-module whose injective hull is not locally Artinian.
\end{proposition}

\begin{proof} Let $\rho\in R$ be a non-zero non-unit of $R$, and let $\mathcal{P} = \{P_1, \ldots , P_t \}$ be the set of annihilator primes in $R$ of the $R$-module $R/\rho R$. Thus $\rho R \subseteq P_i$ for all $i = 1, \ldots , t$, and the inverse image in $R$ of every minimal prime of $R/\rho R$ is in $\mathcal{P}$, but there may also be further primes in $\mathcal{P}$. We divide the proof in two cases.

\medskip

\noindent{\bf Case 1.} \begin{equation}\label{hijak1} \alpha^j (P_i) \notin \mathcal{P} \textit{ for all } j \in \Z\setminus \{0\}, 1 \leq i \leq t.
\end{equation}
\medskip

Since $R$ is commutative and $\alpha$-simple, $S/(1-\theta)S$ is a simple $S$-module by Lemma \ref{artinian}(b). That is, $V := \rho S/\rho(1-\theta)S$ is a simple submodule of $M:=S/\rho(1-\theta)S$. By Lemma \ref{notartin}, $M$ is not an Artinian $S$-module. We shall prove  that
\begin{equation} \label{crux} V \textit{ is an essential submodule of } M; \end{equation}
that is, the thesis holds for $n=1$ in this case.

Let $p$ be a non-zero element of $M$ and apply  Lemma \ref{form} to write
\begin{equation}\label{pform} p=\sum_{i=l}^{k} r_{\lambda_i}\theta^i+\rho b+\rho (1-\theta)S
\end{equation}
for unique $b\in R$ and coset representatives $r_{\lambda_i}$ for $\rho R$ in $R$. Here, as in Lemma \ref{form}, either the summation part of the expression for $p$ is $0$, so that $p \in V$, or we fix $\ell$ and $k$ so that $r_{\lambda_{\ell}}$ and $r_{\lambda_k}$ are non-zero. In the former case we define the \emph{length} $\ell (p)$ of $p$ to be $0$, and in the latter case we define the number of $i$ for which $r_{\lambda_i}$ is non-zero to be the \emph{length} $\ell (p)$ of $p$. The strategy of the proof is to show that, if $\ell (p) > 0$, then a non-zero element of strictly shorter length can be found in the submodule $pS$ of $M$.

\bigskip

\noindent{\bf Sublemma 1:} If $\ell(p) > 1,$ then $pS$ contains a non-zero element $\widehat{p}$ with $0 < \ell(\widehat{p}) < \ell (p)$.

\noindent {\it Proof of Sublemma 1.} Amongst those terms in the expression (\ref{pform}) for $p$ with $r_{\lambda_i} \neq 0$, choose $r_{\lambda_j}$ such that $r_{\lambda_j}R + \rho R/\rho R$ has a maximal annihilator ideal $Q$ which has minimal height amongst all the maximal annihilators in $R$ of the non-zero $R$ modules in the list $\{r_{\lambda_i}R + \rho R/\rho R : \ell \leq i \leq k \}$. Let $r \in R$ be such that
$$ \mathrm{Ann}_R (X) = Q $$
for every non-zero $R$-submodule $X$ of $r_{\lambda_j}rR + \rho R/\rho R$. Now multiply $p$ by $\alpha^{-j} (r)$, to get
$$\begin{aligned}\label{onedone}  p_1 \, &:= \,p\alpha^{-j}(r)\\ &= \; r_{\lambda_j}r\theta^j + \sum_{i \neq j} r_{\lambda_i}\alpha^{i-j}(r)\theta^i + \rho b\alpha^{-j}(r) + \rho (1 - \theta) S.
\end{aligned}$$
Observe that, since we are assuming $\ell (p) > 1$, the summation over $i$ in the above expression is not empty. If one or more of the terms appearing under the summation symbol features $r_{\lambda_i}\alpha^{i-j}(r) \in \rho R$, then the proof of the sublemma is complete, taking $\widehat{p} := p_1$. Suppose then that this is not the case, and choose $w$ , $\ell \leq w \leq k, \, w \neq j$, such that $r_{\lambda_w}\alpha^{w-j} (r)R$ has a maximal annihilator ideal $J$ which is maximal amongst annihilator ideals of non-zero elements of the $R$-modules in the collection $\{ r_{\lambda_i}\alpha^{i-j}(r)R + \rho R/\rho R : i \neq j \}$. Choose $\widehat{r} \in R$ such that
\begin{equation}\label{tense} \mathrm{Ann}_R (r_{\lambda_w}\alpha^{w-j}(r)\widehat{r} + \rho R) = J.
\end{equation}
Now multiply $p_1$ by $\alpha^{-w}(\widehat{r})$ to get
\begin{align*} \label{payoff} p_2 \, &:= \, p_1\alpha^{-w}(\widehat{r})\\ &= r_{\lambda_j}r\alpha^{j-w}(\widehat{r})\theta^j + \sum_{i \neq j,w} r_i\alpha^{i-j}(r)\alpha^{i-w}(\widehat{r})\theta^i +  r_{\lambda_w}\alpha^{w-j}(r)\widehat{r}\theta^w + \rho b\alpha^{-j}(r)\alpha^{-w}(\widehat{r}) + \rho (1 - \theta) S.
\end{align*}
If $r_{\lambda_j}r\alpha^{j-w}(\widehat{r}) \in \rho R$ then the sublemma is proved. So, suppose on the other hand that
\begin{equation} \label{clash} r_{\lambda_j}r\alpha^{j-w}(\widehat{r}) \notin \rho R.
\end{equation}
Observe that $\alpha^{j-w}(J)$ cannot be strictly contained in $Q$, since
$$ \mathrm{height}(Q) \leq \mathrm{height}(J) = \mathrm{height}(\alpha^{j-w}(J)),$$
thanks to our choice of $Q$. Moreover, thanks to (\ref{hijak1}),we cannot have $\alpha^{j-w}(J) = Q$. Thus it is possible to choose $x \in J$ with $\alpha^{j-w}(x) \notin Q$. Then
$$ \widehat{p} \, := \, p_2 \alpha^{-w}(x) \, \in \, pS \subseteq M$$
satisfies
$$ 0 < \ell(\widehat{p}) < \ell(p), $$
and the sublemma is proved.

\bigskip

\noindent {\bf Sublemma 2:} Let $p \in M$ be written in normal form (\ref{pform}), with $\ell(p) = 1$ - that is, $\ell = k$ and $r_{\lambda_k} \neq 0.$ Then there exists $u \in R$ such that $0 \neq pu \in V.$

\noindent {\it Proof of Sublemma 2.} Simplifying notation slightly, let $p$ have coset representative
\begin{equation}\label{drip} r_i\theta^i + \rho b \in S \setminus \rho (1 - \theta)S,
\end{equation}
where $0 \neq r_i \in R \setminus \rho R$, $b \in R$, and $i \geq 0.$ We have to show that there exists $u \in R$ such that
\begin{equation}\label{high} (r_i\theta^i + \rho b)u \in \rho S \setminus \rho (1 - \theta)S. \end{equation}
We prove that (\ref{high}) is true with $u = \alpha^{-i}(\rho).$ For $(r_i\theta^i + \rho b)\alpha^{-i}(\rho)$ clearly lies in $\rho S$. On the other hand,
\begin{align*}  & \qquad (r_i\theta^i + \rho b)\alpha^{-i}(\rho) \in \rho(1 - \theta) S \\
& \Leftrightarrow \rho r_i \theta^i + \rho b \alpha^{-i}(\rho) \in \rho(1 - \theta) S \\
& \Leftrightarrow r_i \theta^i +  b \alpha^{-i}(\rho) \in (1 - \theta) S \\
& \Leftrightarrow (\theta^i - 1)\alpha^{-i}(r_i) + \alpha^{-i}(r_i) + b \alpha^{-i}(\rho) \in (1 - \theta)S \\
& \Leftrightarrow \alpha^{-i} (r_i) = -b\alpha^{-i}(\rho) \\
& \Leftrightarrow r_i = - \rho \alpha^i (b), \\
\end{align*}
where the penultimate equivalence follows from Lemma \ref{artinian}(a). But $r_i = - \rho \alpha^i (b)$
is a contradiction to the initial hypothesis that $r_i \notin \rho R$, so that (\ref{high}) is true and Sublemma 2 is proved.

Thus (\ref{crux}) follows at once from the two sublemmas, and $V=S/(1-\theta)S$ is a simple $S$-module with an injective hull that is not locally Artinian.

\medskip

\noindent{\bf Case 2.} Assume that (\ref{hijak1}) does not hold. Note that if $R$ is $\alpha$-simple, then $R$ is also $\alpha^m$-simple for any $m\in\N$. Indeed if $I$ is $\alpha^m$-stable proper ideal of $R$, then $I\alpha(I)\ldots\alpha^{m-1}(I)=0$ as it is an $\alpha$-stable proper ideal, and so $I = (0)$ since $R$ is a domain. Hence also every non-zero prime ideal has an infinite $\alpha^m$-orbit. Choose $n\in \N$ such that the conclusion of Lemma \ref{orbital} holds for $\alpha$ and $\mathcal{P}$. By the above we can apply Case 1 to $S'=R[\theta^n; \alpha^n]$ and obtain that $V=S'/(1-\theta^n)S'$ is a simple $S'$-module with an  injective hull that is not locally Artinian.

As an $S'$-module, $S$ is free with a basis formed by normal elements, $1, \theta, \ldots, \theta^{n-1}$. The $S$-module $V\otimes_{S'} S\simeq S/(1-\theta^n)S$ is, as an $S'$-module, isomorphic to
$$\bigoplus_{i=0}^{n-1} V^{\alpha^i}$$
where $V^{\alpha^i}$ is the $S'$-module defined by taking $V$ as the underlying Abelian group and setting $v\cdot s=v\alpha^i(s)$, where $\alpha^i$ is the automorphism of $S'$ induced by $\alpha^i$. Since $V$ is simple as $S'$-module, so is  each $V^{\alpha^i}$. Hence $V\otimes_{S'} S$ is a finite length $S'$-module. Therefore $V\otimes_{S'} S$ is a finite length $S$-module.

Since $S$ is free as an $S'$-module, $V\otimes_{S'}S\leq E_{S'}(V)\otimes_{S'}S$. Write $\leq_e$ for ``is an essential submodule of'', and observe that $V\otimes_{S'}S\leq_e E_{S'}(V)\otimes_{S'}S$ as $S'$-modules, and hence \emph{a fortiori} as $S$-modules. So $$E_{S'}(V)\otimes_{S'} S \leq_e E_S(V\otimes_{S'} S).$$

Since $E_1:=E_{S'}(V)\otimes_{S'} S$ is contained in $E_S(V\otimes S)$ and is not locally Artinian, there is a finitely generated $S'$-submodule $M$ of $E_1$ which is not $S'$-Artinian. The finitely generated $S$-submodule $MS$ of $E_1$. Then $MS$ is finitely generated also as an $S'$-module, and so, by \cite[Theorem 4]{FormanekJategaonkar} if $MS$ were $S$-Artinian, then $MS$ would be $S'$-Artinian. This is a contradiction, since $M \leq MS$ as $S'$-submodule. So $MS$ is not $S$-Artinian and $E_S (V \otimes S)$ is not a locally Artinian $S$-module, as claimed.

\end{proof}

In fact it can be shown, using an argument based on \cite[Exercise 31 on page 112]{Stenstrom} that $E_{S'}(V)\otimes_{S'} S=E_S(V\otimes S)$ but we do not give details here since we do not need that fact.

\section{Necessary and sufficient conditions for $(\diamond)$}\label{necsuff}

In handling primitive skew polynomial rings our approach is to strengthen the $\alpha$-special property, guaranteed by Theorem \ref{JurekLeroy} for such a primitive ring, to the stronger $\alpha$-simple condition. This can be achieved by localising at the smallest Ore set $\mathcal{A}$ of $S$ containing the $\alpha$-special element. However, to apply Proposition \ref{simplenotdiamond1} after this localisation, we need to exclude the possibility that $R\mathcal{A}^{-1}$ is a field. To achieve this, we need to assume that $R$ is ``sufficiently big'' in terms both of the height and the ``width'' of its lattice of prime ideals, and that it contains an uncountable field. The technicalities for this are provided by the lemmas of $\S$\ref{prelim}. The key results are then Proposition \ref{old2.1} and Theorem \ref{primitivecase}.

\subsection{Localisation lemmas}\label{prelim}

\begin{lemma}
\label{kenlem1} Let $T$ be a ring and ${\mathcal A}$ a multiplicatively closed Ore subset of regular elements with $1\in {\mathcal A}$. If there exists a right $T$-module $V$ such that $E_{T{\mathcal A}^{-1}}(V\otimes_T T{\mathcal A}^{-1})$ is not locally Artinian as a $T{\mathcal A}^{-1}$-module, then $E_T(V/\tau(V))$ is not locally Artinian, where $\tau(V)$ denotes the $\mathcal A$-torsion submodule of $V$.
\end{lemma}

\begin{proof} Suppose that $E_{T{\mathcal A}^{-1}}(V\otimes T{\mathcal A}^{-1})$ is not locally Artinian as a $T{\mathcal A}^{-1}$-module. In particular $V$ is not ${\mathcal A}$-torsion. Let $\tau(V)$  be the ${\mathcal A}$-torsion submodule of $V$, then $\frac{V}{\tau(V)}\otimes T{\mathcal A}^{-1}\simeq V\otimes T{\mathcal A}^{-1}$. Thus replacing $V$ by $V/\tau(V)$ we may assume that $\tau(V)=0$.

It is easy to check that $E_{T{\mathcal A}^{-1}}(V\otimes_T T{\mathcal A}^{-1})$ is an essential extension of $V$ as a $T$-module. In particular this implies that
$$E_{T{\mathcal A}^{-1}}(V\otimes_T T{\mathcal A}^{-1}) \leq E_T(V).$$

Now let $e_1,\ldots, e_n\in E_{T{\mathcal A}^{-1}}(V\otimes_T T{\mathcal A}^{-1})$ be such that $\displaystyle W:=\sum_{j=1}^m e_j T{\mathcal A}^{-1}$ contains an infinite descending chain
\begin{eqnarray}\label{down}
\ldots \subset W_{i+1}\subset W_i\subset \ldots \subset W_2\subset W_1=W
\end{eqnarray}
of $T{\mathcal A}^{-1}$-submodules of $W$. For all $i$, $W_i=(W_i\cap  \sum_{j=1}^m e_j T)T{\mathcal A}^{-1}$, so (\ref{down}) yields  an infinite strictly descending chain of $T$-submodules of $\displaystyle\sum_{j=1}^m e_j T$, as required.
\end{proof}

The example (already featuring in $\S$\ref{prim}) of $R = k[[X]]$ and $\mathcal{A} = \{X^i : i \geq 0 \}$ should be borne in mind in conjunction with the next lemma.

\begin{lemma}\label{kdim1}
Let $R$ be a commutative Noetherian domain which is an algebra over an uncountable field, but is not itself a field. Suppose that there exists a countable multiplicatively closed subset ${\mathcal A}$ of $R\setminus \{0\}$ such that $R{\mathcal A}^{-1}$ is the quotient field of $R$. Then $R$ has Krull dimension 1 and $\mathrm{Spec}(R)$ is countable.
\end{lemma}
\begin{proof}
Assume $R$ and ${\mathcal A}$ are as above. By \cite[Proposition 2.5]{SharpVamos} every Noetherian ring containing an uncountable field has the countable prime avoidance property. It follows from \cite[Theorem 3.8]{KaramzadehMoslemi} that $\mathrm{Spec}(R)$ is countable and that each non-zero prime ideal is maximal.
\end{proof}

\subsection{$(\diamond)$ when $S$ is primitive}\label{what}

Recall that the multiplicatively closed $\alpha$-invariant Ore subset ${\mathcal A}$ of $S$ generated by an $\alpha$-special element $a$ of $R$ was defined in Remark \ref{torfree}.

\begin{proposition}\label{old2.1}

Let $R$ be a commutative Noetherian domain, $\alpha$ an automorphism of $R$.   Suppose that $S=R[\theta; \alpha]$ is primitive, with $\alpha$-special element $a$ and associated Ore subset ${\mathcal A}$. Suppose that $R{\mathcal A}^{-1}$ is not a field. Then $S$ does not satisfy ($\diamond$).
\end{proposition}

\begin{proof}
The existence of the element $a$ is guaranteed by Theorem 4.2. Note that $R{\mathcal A}^{-1}$ is $\alpha$-simple and $S{\mathcal A}^{-1}=R{\mathcal A}^{-1}[\theta; \alpha]=R{\mathcal A}^{-1}[a\theta; \alpha]$. By Proposition 4.10 there is $n\in \N$ such that
$$E_{S{\mathcal A}^{-1}}\left(S{\mathcal A}^{-1}/(1-(a\theta)^n\right)S{\mathcal A}^{-1})$$
is not locally Artinian.

Set $S'=R[\theta^n; \alpha^n]$. Since for any $\alpha^n$-stable ideal $I$ of $R$, the ideal $I\alpha(I)\ldots\alpha^{n-1}(I)$ is $\alpha$-stable, it is clear that $R$ is $\alpha^n$-special with the special element $b=a\alpha(a)\alpha^2(a)\ldots\alpha^{n-1}(a)$. Thus $S'$ is primitive and by Proposition 4.3, $V=S'/(1-b\theta^n)S'=S'/(1-(a\theta)^n)S'$ is simple as a $S'$-module since $b\theta^n=(a\theta)^n$.

By similar arguments to the ones at the end of the proof of Proposition \ref{simplenotdiamond1},  $W:=V\otimes S=S/(1-(a\theta)^n)S$ is of finite length as $S$-module. By Lemma 5.1 the injective hull of the finite length module $W/\tau(W)$ is not locally Artinian and the result follows.
\end{proof}

We can now give necessary and sufficient conditions for $S = R[\theta;\alpha]$ to satisfy $(\diamond)$ when $S$ is primitive. Unfortunately, however, these two conditions do not together cover all possibilities.

\begin{theorem}\label{primitivecase} Let $R$ be a commutative Noetherian ring, $\alpha$ an automorphism of $R$, and let $S = R[\theta;\alpha]$. Suppose that $S$ is primitive.
\begin{enumerate}
\item If $R$ has Krull dimension 0 then $S$ satisfies $(\diamond)$.
\item Suppose that $R$ contains an uncountable field. Suppose also that either $R$ has Krull dimension at least 2, or $\mathrm{Spec}(R)$ is uncountable. Then $S$ does not satisfy $(\diamond)$.
\end{enumerate}
\end{theorem}

\begin{proof} (a) When $R$ has Krull dimension 0, $S$ has Krull dimension 1 by \cite[Theorem 15.19]{GoodearlWarfield}. Since $S$ is a prime Noetherian ring of Krull dimension 1, $(\diamond)$ follows easily (see for instance \cite[Proposition 5.5]{Musson80}).

(b) Suppose that $R$ contains an uncountable field, and that $R$ has Krull dimension 2 or more, or has uncountably many prime ideals. Let $Q$ be a minimal prime ideal of $R$, let $n \in\mathbb{N}$ be such that the minimal primes of $R$ are $\{\alpha^i (Q) : 0 \leq i \leq n-1\}$, and set $T := R[\theta^n ; \alpha^n]$. By Lemma \ref{npower}(b) $T/QT$ is also primitive, and by Lemma \ref{power}(b) it is enough to show that $T/QT$ does not satisfy $(\diamond)$. In other words, in proving (b), we can pass to $(R/Q)[\theta^n; \alpha^n]$, observing that $R/Q$ will like $R$ contain an uncountable field, and have Krull dimension at least 2 or have uncountably many prime ideals, just as $R$ does. That is, we may assume additionally that $R$ is a domain in proving (b).

Let $a$ be the $\alpha$-special element of $R$ guaranteed by Theorem \ref{JurekLeroy}, with $\mathcal{A}$ the $\alpha$-invariant Ore set it generates. By Lemma \ref{kdim1} $R\mathcal{A}^{-1}$ cannot be the quotient field of $R$. Proposition \ref{old2.1} now implies that $S$ does not satisfy $(\diamond)$.
\end{proof}

It is straightforward to show that in Theorem \ref{primitivecase} under the hypothesis  (a), the ring  $R$  is a finite direct sum of isomorphic fields.

\section{Property $(\diamond)$ when $R$ is affine}

The gap between parts (a) and (b) of Theorem \ref{primitivecase} can be closed so as to determine completely the occurrence of $(\diamond)$ for $S = R[\theta ; \alpha]$, provided $R$ contains an uncountable field and has no $\alpha$-invariant factors of Krull dimension 1 whose spectrum is countable. In particular, this allows us to completely settle the matter when $R$ is an affine algebra over an uncountable field, as follows.

\begin{theorem}\label{affine}
Let $k$ be an uncountable field and $R$  an affine $k$-algebra, and let $\alpha$ be a $k$-algebra automorphism of $R$. Let $S = R[\theta ; \alpha]$. Then the following are equivalent:
\begin{enumerate}
 \item $S$ satisfies $(\diamond)$;
 \item all simple $S$-modules are finite dimensional $k$-vector spaces.
\end{enumerate}
\end{theorem}
\begin{proof}
(b)$\Rightarrow$(a) This follows immediately from Propositions \ref{ARsep} and \ref{sslc} and does not require that $k$ is uncountable.

(a)$\Rightarrow$(b) Suppose that $S$ satisfies $(\diamond)$. Let $V$ be a simple $S$-module and $P=\mathrm{Ann}_S(V)$.

If $\theta\in P$, then $P/\theta S$ is a primitive ideal of the commutative affine $k$-algebra $R$, so $V$ is finite dimensional over $k$, thanks to the Nullstellensatz.

 Suppose next that $(P\cap R)S\subset P$ and $\theta\notin P$. Then \cite[Theorem 4.3]{Irving79}  implies that $S/(P\cap R)S$ satisfies a polynomial identity. Since $S/(P\cap R)S$ is also by hypothesis and construction a   Noetherian affine $k$-algebra, $V$ is finite dimensional over $k$ by Kaplansky's theorem, \cite[Theorem I.13.3]{BrownGoodearl}.

Suppose finally that $(P\cap R)S=P$ and let $Q$ as usual denote a minimal prime over $P\cap R$  in $R$. If $R/P \cap R$ has Krull dimension 0, it is a finite direct sum of copies of the field $R/Q := F$,  by Lemma \ref{stable}(d). Hilbert's Nullstellensatz ensures that $F$ is finite dimensional over $k$, so that some finite power, say $n$, of $\alpha$ not only fixes $Q$, but then also induces the identity on $F$. Hence, $S/P$ is a finite (free) module over the commutative ring $(R/P \cap R)[\theta^n]$, so $S/P$ satisfies a polynomial identity and Kaplansky's theorem applies. Suppose on the other hand that $R/P \cap R$ has Krull dimension at least 1. Then, for example using Lying Over and the Noether normalisation theorem \cite[Proposition 4.15 and $\S$8.2.1, Theorem A1]{E}, $\mathrm{Spec}(R/Q)$ is uncountable since $k$ is uncountable. By Theorem \ref{primitivecase}(b), this contradicts property $(\diamond)$ for $S$. So no such primitive ideals $P$ can exist in $S$ and (b) follows.
\end{proof}

Notice that $(b)\Rightarrow(a)$ of this theorem is valid, with the same proof, for an arbitrary commutative Noetherian $k$-algebra over any field. Indeed, if we change statement $(b)$ to
$$ (b') \quad S/\mathrm{Ann}_S(V) \textit{ is Artinian for all simple } S\textit{-modules } V, $$
then a small adjustment to the argument confirms that $(b')\Rightarrow(a)$ is true for all commutative Noetherian coefficient rings $R$.

Let $k$ be any field, let $R$ be a commutative affine $k$-algebra and let $\alpha$ be a $k$-algebra automorphism of $R$. In the light of Theorem \ref{affine} it is natural to ask exactly what conditions on $R$ and $\alpha$ are required for (b) of Theorem \ref{affine} to hold. This appears to be quite a subtle question, which we discuss further in  \S 7. First, we give here an important special case in which a complete answer is available:

\begin{proposition} \label{poly}Let $k$ be a field, $t$ a positive integer, $V$ a vector space over $k$ with basis $\{x_1, \ldots , x_t\}$ and $\alpha \in \mathrm{GL}(t,k)$ an automorphism of $V$. Let $R = k[x_1, \ldots , x_t]$, so that $\alpha$ induces a $k$-algebra automorphism of $R$, also denoted by $\alpha$. Then $S := R[\theta; \alpha]$ satisfies $(\diamond)$ if and only if the order $|\alpha|$ of the automorphism $\alpha$ is finite.
\end{proposition}
\begin{proof} Suppose $|\alpha| = n < \infty$. Since the commutative Noetherian ring $R[\theta^n]$ satisfies $(\diamond)$, Lemma \ref{power}(a) implies that S satisfies $(\diamond)$.

For the converse, suppose that $S$ satisfies $(\diamond)$. Let $L$ be a finite extension field of $k$ such that the Jordan Normal Form of $\alpha$ exists in $GL(t,L)$. By Proposition \ref{reduction} $L \otimes_k S$ satisfies $(\diamond)$, so we may replace $R$ by $\widehat{R} = L[x_1, \ldots , x_t]$ and $S$ by $\widehat{S} := \widehat{R}[\theta; \alpha]$ in proving that $|\alpha| < \infty$. Changing the basis of $V$ as necessary, we can assume that $\alpha$ has Jordan Normal Form with respect to the ordered basis $\{x_1, \ldots , x_t\}$. We show first that
\begin{equation}\label{evalue} \textit{every generalised eigenvalue of } \alpha \textit{ is a root of unity in } L.
\end{equation}
Suppose that (\ref{evalue}) is false, so there exists $i$ such that the generalised eigenvalue $a_i$ of $x_i$ is not a root of $1$.
Choose $i$ and $j\leq i$ so that $\{x_j,\ldots,x_i\}$ is a basis of the Jordan block of $\alpha$ for the eigenvalue $a_i$.
Then it is clear that
$$ P \quad := \quad \sum_{\ell \neq i} x_{\ell}\widehat{R} $$
is an $\alpha$-invariant ideal of $\widehat{R}$. Thus $P\widehat{S} \lhd \widehat{S}$, with
$$ \widehat{S}/P\widehat{S} \cong (\widehat{R}/P)[\theta; \overline{\alpha}] = L[x_i][\theta ; \overline{\alpha}],$$
slightly abusing notation by writing $x_i$ for the image of that element of $\widehat{R}$ in $\widehat{R}/P$, and where $\overline{\alpha}$ is the $L$-algebra automorphism of $L[x_i]$ defined by $\overline{\alpha}(x_i) = a_i x_i.$ Thus $\widehat{S}/P\widehat{S}$ is a quantum plane with parameter $a_i$ and does not satisfies $(\diamond)$ by \cite[Theorem 3.1]{CarvalhoMusson}. This contradicts the hypothesis that $S$ and therefore $\widehat{S}$ satisfy $(\diamond)$, so (\ref{evalue}) is proved. (This can also be seen with the help of Proposition \ref{old2.1}. Indeed one checks easily that $x_i$ is an $\overline{\alpha}$-special element of $\widehat{R}/P$, with corresponding $\overline{\alpha}$-invariant Ore set $\mathcal{A}$ in $\widehat{R}/P$ equal to $\{a_i^m x_i^n : m \in \mathbb{Z}, \, n \in \mathbb{N} \}$. Hence $\widehat{S}/P\widehat{S}$ is primitive and since $(\widehat{R}/P)\mathcal{A}^{-1}$ is not a field, we conclude from the proposition that $\widehat{S}/P\widehat{S}$ does not satisfy $(\diamond)$.)

Notice also that this completes the proof of the proposition in the case where $k$, or equivalently $L$, has positive characteristic, since then the matrix of $\alpha$ lies in $GL(t, E)$ for some finite field $E$.

Now suppose $L$ has characteristic $0$, and that $\alpha$ is \emph{not} diagonalisable, so that there is a Jordan block of $\alpha$ of cardinality greater than $1$. Let $n$ be the least common multiple of the orders of the generalised eigenvalues of $\alpha$. Replacing $\widehat{S}$ by $\widehat{R}[\theta^n; \alpha^n]$, using Lemma \ref{power}(a), we may assume that the only generalised eigenvalue of $\alpha$ is 1. We aim for a contradiction to the hypothesis that $\widehat{S}$ satisfies $(\diamond)$. Rearranging the order of the basis vectors of $V$ if need be, we can assume that there is a block of $\alpha$ with basis $\{x_m, \ldots , x_t\}$, for some $m < t$. Consider now the $\alpha$-invariant ideal of $\widehat{R}$,
$$ I \quad := \quad \sum_{\ell = 1}^{t-2}x_{\ell}\widehat{R} + (x_{t-1} - 1)\widehat{R}. $$
Passing to $\widehat{S}/I\widehat{S} \cong (\widehat{R}/I)[\theta; \alpha] \cong L[x_t][\theta; \overline{\alpha}]$, we may assume that
\begin{equation}\label{Muss} \overline{\alpha} : x_t \mapsto x_t + 1.
\end{equation}
 One readily checks that $\widehat{S}/I\widehat{S}$ is isomorphic to the enveloping algebra of the 2-dimensional solvable non-abelian Lie algebra over a field of characteristic 0, so it does \emph{not} satisfy $(\diamond)$, by \cite{Musson82}; alternatively, since $L$ has characteristic 0, $\widehat{R}/I$ is $\alpha$-simple and one can appeal directly to Proposition \ref{simplenotdiamond1} to get a contradiction to property $(\diamond)$ for $\widehat{S}$. Therefore, $\alpha$ is diagonalisable; coupling this with (\ref{evalue}), the proposition is proved.
\end{proof}

\section{Examples and discussion}\label{examples}

We discuss in this section property $(\diamond)$ for $S = R[\theta; \alpha]$ when the base ring $R$ is an integral domain and an affine $k$-algebra over a field $k$, and $\alpha$ is a $k$-algebra automorphism of $R$. In the light of Theorem \ref{affine} and Proposition \ref{poly}, this amounts to examining the interconnections between the following four hypotheses:
\begin{enumerate}
\item[$(\bullet)$] $S$ satisfies $(\diamond)$.
\item[$(\bullet)$] Every simple $S$-module has finite dimension over $k$.
\item[$(\bullet)$] $S$ satisfies a polynomial identity.
\item[$(\bullet)$] $\alpha$ has finite order.
\end{enumerate}
We have seen in $\S 6$ that the first two of these are equivalent, that the third and fourth are equivalent, by \cite{DamianoShapiro}, and that the third implies the second, by Kaplansky's theorem, \cite[Theorem I.13.3]{BrownGoodearl}. Moreover all four are equivalent when $\alpha$ is a linear automorphism of a polynomial algebra, by Proposition \ref{poly}.

We first consider, in $\S$\ref{1poly} and $\S$\ref{2poly}, the application of Theorem \ref{affine} to the algebras $S_{k,n,\alpha}:= k[x_1, \ldots , x_n][\theta; \alpha]$, where $k$ is a field, $n$ is a positive integer, and $\alpha$ is as a $k$-algebra automorphism of $k[x_1, \ldots , x_n]$. Here, we will only consider $n=1$ and $n=2$, but even for these small values of $n$ the situation turns out to be surprisingly delicate. The representation theory of $S_{k,n,\alpha}$ likely has close interactions with the dynamical properties of $\alpha$, as studied for example in \cite{FM} and subsequent works; see for instance \cite{Raluca}. Even for the case $n=2$ we are unable to determine whether the first bullet point above is equivalent to the third and fourth.

We then turn in $\S$\ref{groupring} to algebras $S=R[\theta;\alpha]$ occurring as subalgebras of group algebras $kG$ of torsion-free polycyclic groups $G$, where $k$ is algebraic over a field of $p$ elements, $p$ prime. Here, thanks to deep results on the representation theory of these group algebras, we can easily construct many examples where the first two of the above bullet points hold, but the second two do not.

It remains to the best of our knowledge an open question at this point whether the four bullet points above are equivalent when $k$ has characteristic 0. In the final subsection, $\S$\ref{Jordan}, we describe an example due to Jordan (cf.\cite{Jordan}), which would, if correct, show that there is no equivalence of the four bullet points for affine algebras in characteristic 0; but unfortunately there appears to be a gap in the analysis of this example in \cite{Jordan}.

\subsection{$(\diamond)$ for $S_{k,1,\alpha}$}\label{1poly}
Here $k$ can be an arbitrary field. As is easy to confirm, the group $A$ of $k$-algebra automorphisms of $k[x]$ consists of the affine automorphisms, mapping $x$ to $\beta x + \gamma$, for $\beta,\gamma \in k$ with $\beta$ non-zero.

\begin{proposition}\label{dimone} With the above notation, the following are equivalent.
\begin{enumerate}
\item $S_{k,1,\alpha}$ satisfies $(\diamond)$;
\item every simple $S_{k,1,\alpha}$-module is finite dimensional over $k$;
\item $S_{k,1,\alpha}$ is a finite module over its centre;
\item either (i) $k$ has characteristic 0, $\beta$ is a root of unity and $\gamma = 0$; or (ii) $k$ has positive characteristic $p$ and $\beta$ is a root of unity.
\end{enumerate}
\end{proposition}
\begin{proof} The implications (d)$\Rightarrow$(c)$\Rightarrow$(b)$\Rightarrow$(a) are standard. Indeed suppose (d) holds. Then $\alpha$ is of finite order and by Noether's Theorem \cite[Theorem 2.3.1]{LSmith}, (c) follows easily. Moreover $S$ satisfies a polynomial identity and by Kaplansky's Theorem \cite[Theorem I.13.3]{BrownGoodearl} we have that (c)$\Rightarrow$(b). The implication (b)$\Rightarrow$(a) follows as in Theorem \ref{affine}.

(a)$\Rightarrow$(d) Suppose that $S_{k,1,\alpha}$ satisfies $(\diamond)$.

First, assume that $k$ has characteristic 0. Then, if $\gamma$ is not 0 we can argue as in the argument to deal with (\ref{Muss}) in the proof of Theorem \ref{affine}, to get a contradiction. So $\gamma = 0$ and we are in the quantum plane setting which arose in proving (\ref{evalue}). By \cite[Theorem 3.1]{CarvalhoMusson} it follows that $\beta$ is a root of unity.

Finally, suppose that $k$ has positive characteristic $p$. Replacing $S_{k,1,\alpha}$ by $T := k[x][\theta^p; \alpha^p]$ and appealing to Proposition \ref{reduction}(a), we may assume that $\gamma = 0$. So $T$ is a quantum plane and $\beta^p$ is a root of unity as above, completing the proof.
\end{proof}

\subsection{$(\diamond)$ for $S_{\C,2,\alpha}$}\label{2poly}
Consider here $S_{\C,2, \alpha} = \C[x,y][\theta; \alpha]$. Let $A := \mathrm{Aut}({\C}[x,y])$, the group of $\C$-algebra automorphisms of $R := {\C}[x,y]$, and let $\alpha \in A$. Following \cite[Proposition 1]{MSmith}, $\alpha$ is \emph{triangular} if it is conjugate in $A$ to one of the following types of automorphism:
\begin{enumerate}
 \item[(i)] $x\mapsto \lambda x; \quad y\mapsto \mu y, \quad \lambda, \mu \in \C\setminus \{0\}$;
 \item[(ii)] $x\mapsto \lambda x; \quad y\mapsto y+c,\quad \lambda,c\in \C\setminus \{0\} $;
 \item[(iii)] $x\mapsto \lambda x+ \displaystyle\sum_{\{i:\lambda=\mu^i\}} \eta_i y^i; \quad y\mapsto \mu y,\quad \eta_i\in{\C}, \,  \lambda, \mu \in \C\setminus \{0\}$.
\end{enumerate}

Note that if $\xi \in A$, then $\xi$ extends to an isomorphism from $R[\theta;\alpha]$ to $R[\theta; \xi\alpha\xi^{-1}]$. Thus, throughout the arguments below, we may where convenient replace a hypothesis that $\alpha$ is conjugate to an element of type (j) by the hypothesis that $\alpha$ belongs to type (j). A similar comment applies also to the later proofs in this subsection.

\begin{lemma}\label{triangular} With the notation as above, suppose that $\alpha$ is triangular. Then the following are equivalent:
\begin{enumerate}
\item $S_{\C,2,\alpha}$ satisfies $(\diamond)$;
\item $\alpha$ is conjugate to an automorphism of type (i), with $\lambda$ and $\mu$ both roots of unity;
\item $\alpha$ has finite order;
\item $S_{\C,2,\alpha}$ is a finite module over its centre;
\item every simple $S_{\C,2,\alpha}$-module is a finite dimensional $k$-vector space.
\end{enumerate}
\end{lemma}
\begin{proof}

The implication (b)$\Rightarrow$(c) is clear and  (c)$\Rightarrow$(d)$\Rightarrow$(e) follow as in the first paragraph of the proof of Proposition \ref{dimone}. Finally, (e)$\Rightarrow$(a) is a special case of (b)$\Rightarrow$(a) of Theorem \ref{affine}.

To prove (a)$\Rightarrow$(b) suppose that $S_{\C,2,\alpha}$ satisfies $(\diamond)$.  If $\alpha$ is of type (ii),  then $xS_{\C,2,\alpha}$ is an ideal of $S_{\C,2,\alpha}$ and $S_{\C,2,\alpha}/xS_{\C,2,\alpha}$ is isomorphic to ${\Bbb C}[y][\theta; \beta]$ where $\beta(y)=y+c$ for $c\neq 0$. By Proposition \ref{dimone} and Proposition \ref{reduction}(c), $S_{\C,2,\alpha}$ does not satisfy $(\diamond)$, contradicting our hypothesis.

Suppose alternatively that $\alpha$ is of type (iii) but not type (i). Since $yS_{\C,2,\alpha}$ is an ideal of $S_{\C,2,\alpha}$ and $S_{\C,2,\alpha}/yS_{\C,2,\alpha}$ is isomorphic to a quantum plane with parameter $\lambda$.   By Propositions \ref{dimone} and \ref{reduction}(c), $\lambda$ is a root of unity. Since $S_{\C,2,\alpha}$ is by hypothesis \emph{not} type (i), this forces $\mu$ also to be a root of unity. By Proposition \ref{reduction} the subalgebra $\C[x,y][\theta^n ; \alpha^n]$ of $S_{\C,2,\alpha}$ also satisfies $(\diamond)$, so we may pass to that subalgebra for a suitable choice of $n$, and thus assume that $\mu=1=\lambda$, (but keep the notation $S_{\C,2,\alpha} = \C[x,y][\theta; \alpha]$). Let $\tau \in \C$ be a root of the polynomial $\displaystyle\sum_{\lambda=\mu^i} \eta_i y^i-1$. Then $(y - \tau)S_{\C,2,\alpha}$ is an ideal of $S_{\C,2,\alpha}$, with $S_{\C,2,\alpha}/(y-\tau)S_{\C,2,\alpha}$ isomorphic to ${\Bbb C}[x][\theta;\overline{\alpha}]$ and $\overline{\alpha}(x)=x+1$. A contradiction once again follows from Propositions \ref{dimone} and \ref{reduction}(c).
\end{proof}

We turn now to a consideration of those automorphisms of $\C[x,y]$ which are not triangular. Again following \cite{MSmith}, an element $\alpha$ of $A$ is called \emph{square} if it is not conjugate to a triangular automorphism, and, following \cite{FM}, $\alpha$ is a \emph{generalised H$\acute{e}$non automorphism} if
$$\alpha(x)= y; \qquad  \alpha(y )= \lambda x + \beta(y)$$
for some $\lambda \in \C\setminus \{0\}$ and $\beta (y) \in \C[y]$, with $\beta(y)$ having degree at least two\footnote{This degree condition is mistakenly omitted from the definition in \cite[page 368]{Jordan}, although it is present in \cite[Definition 2.5]{FM} The condition omitted by \cite{Jordan} is definitely needed for the validity of some of the results in \cite{FM}.}. By \cite[Theorem 2.6]{FM}, every square element of $A$ is conjugate to a product of generalised H$\acute{\mathrm{e}}$non automorphisms\footnote{\cite[Theorem 2.6]{FM} is slightly mis-stated, since the cyclically reduced automorphisms which it treats include all the square automorphisms, but also the affine automorphisms which are \emph{not} also elementary, such as, for example the map $\tau: x\mapsto y, \, y\mapsto x$. See the definitions on pages 68-69 of \cite{FM}. In fact, \cite[Theorem 2.6]{FM} is valid, in the terminology of that paper, for the elements of $A$ having degree at least 2. }.

 Let us examine first the prime spectrum of $S_{\C,2,\alpha}$ when $\alpha$ is square. We follow initially in the proof of the following lemma the argument of \cite[proof of Proposition 7.8, first paragraph]{Jordan}.

\begin{lemma}\label{square} Let $\alpha$ be a square automorphism of $\C[x,y]$. Then $\mathrm{Spec}(S_{\C,2,\alpha})$ is the disjoint union of $\mathcal{V}(\theta) := \{ P  \in \mathrm{Spec}(S_{\C,2,\alpha}): \theta \in P \}$, and $\mathcal{C}(\theta) := \mathrm{Spec}(S_{\C,2,\alpha}) \setminus \mathcal{V}(\theta)$, and these sets have the following descriptions:
\begin{enumerate}
\item $\mathcal{V}(\theta)$ is homeomorphic to $\mathrm{Spec}(\C[x,y])$.
\item $\mathcal{C}(\theta)$ is homeomorphic to $\mathrm{Spec}(T_{\C,2,\alpha})$, where $T_{\C,2,\alpha}:= \C[x,y][\theta^{\pm 1}; \alpha]$.
\item $\mathrm{Spec}(T_{\C,2,\alpha})$ is partitioned into the following three disjoint subsets.
\begin{enumerate}
\item[(i)] An uncountable set of co-Artinian maximal ideals $\{M_{j,\lambda}: j \in \mathbb{N}, \lambda \in \C\}$, all of height 2;
\item[(ii)] A countably infinite set of height one mutually comaximal prime ideals $\{P_j : j \in \mathbb{N} \}$, with $$\bigcap_{j \in \mathbb{N}}P_j \quad = \quad (0);$$
\item[(iii)] $(0).$
\end{enumerate}
\item For each $j \in \mathbb{N}$, there exists a positive integer $n_j$ and a finite $\langle \alpha \rangle$-orbit $\{Q_{j,0}, Q_{j,1} = \alpha(Q_{j,0}),\ldots , Q_{j, n_j - 1} = \alpha^{n_j-1}(Q_{j,0})\}$ of maximal ideals of $\C [x,y]$, such that
    $$ P_j \quad = \quad (\bigcap_{\ell = 0}^{n_{j}-1} Q_{j,\ell})T_{\C,2,\alpha}. $$
\item For $j \in \mathbb{N}$, denote the prime ideal $P_j \cap S_{\C,2,\alpha}$ of $S_{\C,2,\alpha}$ by $P_j'$. Thus $P_j' = (\bigcap_{\ell = 0}^{n_{j}-1} Q_{j,\ell})S_{\C,2,\alpha}$. For each $j \in \mathbb{N}$, (writing $\theta$ also for the image of $\theta$ in $S_{\C,2,\alpha}/P_j'$ and in $T_{\C,2,\alpha}/P_j$),
    \begin{eqnarray}\label{equationJ} (S_{\C,2,\alpha}/P_j')[\theta^{-1}] = T_{\C,2,\alpha}/P_j \quad \cong M_{n_j}(\C[\theta^{\pm n_j}]). \end{eqnarray}
    In particular, $T_{\C,2,\alpha}/P_j$ is an Azumaya algebra over its centre $\C[\theta^{\pm n_j}]$, so the maximal ideals of $T_{\C,2,\alpha}/P_j$ are parametrised by $\C$, and are the ideals
    $$ \{ M_{j, \lambda}/P_j : \lambda \in \C \}. $$
    Hence, for all $j \in \mathbb{N}$ and $\lambda \in \C$,
    $$ T_{\C,2,\alpha}/M_{j, \lambda} \quad \cong \quad M_{n_j}(\C). $$
\item The integers $n_j$, for $j \in \mathbb{N}$, are unbounded; more precisely, every sufficiently large prime number occurs amongst the $n_j$.
\end{enumerate}
\end{lemma}
\begin{proof} The partition of $\mathrm{Spec}(S_{\C,2,\alpha})$ into $\mathcal{V}(\theta)$ and $\mathcal{C}(\theta)$ is clear, and (a) and (b) follow at once since $S_{\C,2,\alpha}/\theta S_{\C,2,\alpha} \cong R$ and $T_{\C,2,\alpha}$ is the localization of $S_{\C,2,\alpha}$ with respect to powers of $\theta$.

By \cite[Theorem 3.1]{FM}, $\C[x,y]$ has countably infinitely many maximal ideals with a  finite $\langle \alpha \rangle$-orbit. Enumerating these orbits by the parameter  $j \in \mathbb{N}$, and letting $n_j$ be the size of the $j$th orbit, the integers $n_j$ satisfy (f), by \cite[Corollary 8.6 and note added in proof, page 97]{FM}. Labelling the $j$th finite $\langle \alpha \rangle$-orbit as in (d), and defining the corresponding induced ideal $P_j$ also as in (d), yields a countably infinite set of comaximal prime ideals of $T_{2,\alpha, \C}$ by Lemma \ref{stable}(b), giving the subset (c)(ii) of $\mathrm{Spec}(T_{\C,2,\alpha})$, see also \cite[Theorem III.31]{SamuelZariski}. The subset (c)(iii) is clear.

Let $I$ be any other prime ideal of $T_{2,\alpha,\C}$, so that $I$ is non-zero and $$I \cap \C[x,y] =(I \cap S_{\C,2,\alpha}) \cap \C[x,y]$$ is an $\alpha$-prime ideal of $\C[x,y]$ by Lemma \ref{stable}(c). If $I \cap \C[x,y]$ is co-Artinian, then, by construction, $I$ contains one of the primes $P_j$ constructed above. Otherwise,
\begin{equation}\label{noncase} I \cap \C[x,y] \quad = \quad (0),
\end{equation}
since $\C[x,y]$ has no $\alpha$-stable proper principal ideal, by \cite{L}; see \cite[note added in proof, page 572]{MSmith}.
Assume now that (\ref{noncase}) holds. By \cite[Theorem 4.3]{Irving79}, as $\alpha$ is of infinite order the only prime of $S_{\C,2,\alpha}$, not containing $\theta$, which lies over $(0)$ in $\C[x,y]$ is $(0)$. Thus $(0)$ is the only prime ideal of $T_{\C,2,\alpha}$ not containing one of the ideals $P_j$.

It remains to prove (e), in the course of which the maximal ideals of (c)(i) will be described.
It follows easily from Lemma \ref{stable}(b) that $T_{\C, 2, \alpha}/P_j\quad\cong \quad {\overline R}[\theta^{\pm 1}; \alpha]$ where ${\overline R}$ is the direct sum of $n_j$ copies of $\C$ and $\alpha$ acts on ${\overline R}$ by $\alpha(e_i)=e_{(i+1)\hspace{-0.15cm}\mod\hspace{-0.07cm} n_j}$, where $\{e_i: 1\leq i\leq n_j\}$ is the set of primitive idempotents of ${\overline R}$. Set $n=n_j$, $f=(1-e_1)\theta, a=\theta^{-n+1}$ and $b=\theta^{-1}$. Then: $f^{n-1}=(1-e_1)(1-e_2)\ldots(1-e_{n-1}) \theta^{n-1}=e_n\theta^{n-1}$,   $af^{n-1}=\theta^{-n+1}e_n\theta^{n-1}=e_1$ and  $fb=1-e_1$. Thus $ af^{n-1}+fb=1$ and $f^n=0 $. Therefore, by \cite[Theorem 1.3]{AAR96}, $R[\theta^{\pm 1};\alpha]=M_n(B)$ with $e_{11}=af^{n-1}=e_1$. As $\theta^n$ is central and $e_1\theta ^ke_1=0$, for $0<k<n$, we have $B=e_{11}R[\theta^{\pm 1};\alpha]e_{11}=(e_1R)[\theta^{\pm n}]\simeq\C[\theta^{\pm  n}]$ and (\ref{equationJ}) holds. Now using standard arguments it is easy to complete the proof.
\end{proof}

We can now summarise our (at present incomplete) knowledge about the occurrence of property $(\diamond)$ for the algebras $S_{\C,2,\alpha}$. Recall that, by definition, every element of $\mathrm{Aut}_{\C\mathrm{-alg}}(\C[x,y])$ is either triangular or square.

\begin{theorem}\label{2dim} Let $R = \C[x,y]$, let $\alpha \in A = \mathrm{Aut}_{\C\mathrm{-alg}}(R)$, and let $S_{\C,2,\alpha} = R[\theta;\alpha]$.
\begin{enumerate}
\item Suppose that $\alpha$ is triangular. Then $S_{\C,2,\alpha}$ satisfies $(\diamond)$ if and only if $\alpha$ has finite order if and only if $\alpha$ is conjugate to an element of type (i), with $\lambda, \mu$ roots of unity in $\C$.
\item Suppose that $\alpha$ is square. Then $S_{\C,2,\alpha}$ satisfies $(\diamond)$ if and only if $S_{\C,2,\alpha}$ is not primitive.
\end{enumerate}
\end{theorem}
\begin{proof}(a) This is part of Lemma \ref{triangular}.

(b) Suppose that $S_{\C,2,\alpha}$ is primitive. Then $S_{\C,2,\alpha}$ does not satisfy $(\diamond)$ by Theorem \ref{primitivecase}(b).

Suppose on the other hand that $S_{\C,2,\alpha}$ is not primitive. Then, from the description of the prime spectrum of $S_{\C,2,\alpha}$ in Lemma \ref{square}, it is clear, bearing in mind Kaplansky's theorem \cite[Theorem I.13.3]{BrownGoodearl}, that the only primitive ideals of $S_{\C,2,\alpha}$ are the co-Artinian maximal ideals $M_{j,\lambda}$. Hence, every simple $S_{\C,2,\alpha}$-module is finite dimensional over $\C$. Therefore, $S_{\C,2,\alpha}$ satisfies $(\diamond)$ by Theorem \ref{affine}.
\end{proof}

Of course, Theorem \ref{2dim}(b) begs the following obvious question. Keep the notation of Theorem \ref{2dim}, and suppose that $\alpha$ is a square automorphism. Is $S_{\C,2,\alpha}$ primitive? The simplest square automorphism is perhaps the map $x \mapsto y,\; y\mapsto x + y^2$. We do not even know the answer in this case.

Using the Leroy-Matczuk criterion for primitivity, Theorem \ref{JurekLeroy}, coupled with Lemma \ref{square}, it is not hard to reduce the above question to one about the geometry of the finite orbits of square automorphisms, a question which may be of independent interest. It seems more natural to frame it in terms of points in the plane $\C^2$, rather than in terms of maximal ideals of $R$:

\begin{proposition}\label{orbitcurve} Keep the notation of Theorem \ref{2dim}, with $\alpha$ square. Denote the countably infinitely many finite $\langle \alpha \rangle$-orbits of points in $\C^2$ by $\mathcal{P}_j$, for $j \in \mathbb{N}$. Then $S_{\C,2,\alpha}$ is primitive if and only if a point $(a_j,b_j)\in \mathcal{P}_j$ can be chosen, for every $j \in \mathbb{N}$, such that $\{(a_j,b_j) : j \in \mathbb{N}\}$ lies on a (not necessarily irreducible) affine curve in $\C^2$.
\end{proposition}

We leave the straightforward proof to the reader and refer to \cite[Lemma 2.6(v)]{Jordan} for an  equivalent description of $\alpha$-special rings.

\subsection{Subalgebras of group algebras over absolute fields}\label{groupring}

This family of examples shows that neither of the conditions (\emph{c}) $S$ \emph{satisfies a polynomial identity} nor (\emph{d}) $|\alpha| < \infty$ is implied by conditions (\emph{a}) or (\emph{b}) of Theorem \ref{affine}. Note that (\emph{c}) and (\emph{d}) are equivalent when $R$ is semiprime, by \cite{PascaudValette} or \cite{DamianoShapiro}.

Let $A$ be a free abelian group of finite rank $t$, $(t > 1)$. Write $A$ multiplicatively with free generating set $X_1, \ldots , X_t$. Let $M \in GL_t(\Z)$, and define an automorphism $\alpha$ of $A$ using $M$; that is, for $\mathbf{X}^{\mathbf{n}} = X_1^{n_1}\ldots X_t^{n_t} \in A$, with $\mathbf{n} = (n_1, \ldots , n_t) \in \Z^t$,
$$ \alpha (\mathbf{X}^{\mathbf{n}}) \quad := \quad \mathbf{X}^{M\mathbf{n}},$$
where $\mathbf{n}$ is written as a column vector on the right hand side.

Let $\langle \theta^{\pm 1} \rangle $ be an infinite cyclic group, and form the semidirect product
$$ G\quad := \quad A \rtimes_{\alpha} \langle \theta^{\pm 1} \rangle ,$$
so that $G$ is a torsion-free polycyclic group of Hirsch length $t+1$. Now let $p$ be a prime integer, and let $k$ be a subfield of the algebraic closure of the field of $p$ elements. By a celebrated result of Roseblade, \cite{Roseblade1}, every simple module for the group algebra $kG$ has finite dimension as a $k$-vector space. (Roseblade's result applies to all polycyclic-by-finite groups; the case where the group is nilpotent-by-finite was proved by Philip Hall in 1959 \cite{Hall}.)

Let $S$ be the subalgebra $kA[\theta;\alpha]$ of $kG$. We claim that $S$ inherits from $kG$ the property that all its simple modules are finite dimensional. To see this, let $W$ be a simple $S$-module. Since the powers of $\theta$ form an Ore set in $S$, $W$ is either $\theta$-torsion or $\theta$-torsion free. In the first case, $W\theta = 0$, so that $W$ is a simple module over $S/\theta S \cong kA$. Therefore $\mathrm{dim}_k(W) < \infty$ by Hilbert's nullstellensatz. In the second case, $W\theta = W$, so $W$ carries a structure as $S\langle \theta^{-1}\rangle$-module, that is as $kG$-module. As such, it is necessarily simple, and hence has finite $k$-dimension by Roseblade's theorem \cite{Roseblade1}.

Finally, choosing $M$ to be a matrix of infinite order, we obtain an algebra $S$ which is a skew polynomial $k$-algebra over a Laurent polynomial coefficient algebra, which has property $(\diamond)$ but for which $|\alpha| = \infty$.

\subsection{A subalgebra of a group algebra in characteristic 0}\label{Jordan}

The following example is considered in \cite[7.10-7.14]{Jordan}.

Let $S=R[\theta; \alpha]$, where $R=\C[x,x^{-1},y,y^{-1}]$ and $\alpha\in \mathrm{Aut}_{\C-\mathrm{alg}}(R)$ is defined by $\alpha(y)=x$ and $\alpha(x)=yx^{-1}$. In \cite{Jordan}, Jordan considers both the algebras $S := R[\theta;\alpha]$ and $T := R[\theta^{\pm 1}; \alpha]$. He proves in \cite[Proposition 7.13]{Jordan} that $T$ is primitive. In the same proposition it is claimed that $R$ is not $\alpha$-special, so that, by the Leroy-Matczuk theorem, Theorem 4.2, $S$ is therefore not primitive. As is shown in \cite[Propositions 7.11, 7.12]{Jordan}, every non-zero primitive ideal of $S$ is co-Artinian. Hence, if $S$ itself is not primitive, then every simple $S$-module is finite dimensional, and $S$ satisfies $(\diamond)$ by Theorem \ref{affine}. However, there appears to be a gap in the proof that $R$ is not $\alpha$-special in \cite[Proposition 7.13]{Jordan}, so that the status of $S$ as regards the first two bullet points listed at the start of  $\S  7$ is unknown. We therefore ask: \emph{Is $S$ primitive?}

Observe that, in the light of Theorem \ref{affine} and the above discussion, $S$ satisfies $(\diamond)$ if and only if it is \emph{not} primitive.

\section{Questions}\label{query}

For the convenience of the reader, we gather here a number of open questions arising from this work, some of them previously raised earlier in the text, some appearing here for the first time. As usual, $k$ is a field, and $S$ will denote the skew polynomial ring $R[\theta; \alpha]$, where $R$ is commutative Noetherian and $\alpha$ is an automorphism. Further hypotheses will be added as needed.

First, a question which seeks to remove the gap between the necessary and sufficient conditions in Theorem \ref{first} (= \ref{primitivecase}).
\begin{question}\label{qn1} Suppose that $S$ is primitive and satisfies $(\diamond)$. Must $R$ be a finite direct sum of fields?
\end{question}
We should note that the remaining case of Theorem \ref{primitivecase} is the one where $\alpha$ is  of infinite order, and $R$ is  an $\alpha$-special Noetherian domain  of Krull dimension $1$ with countable spectrum. Moreover  $R{\mathcal A}^{-1}$ is a field, where ${\mathcal A}$, as before, is the smallest Ore set of $S$ containing an $\alpha$-special element. When ${\mathcal A}$ is finitely generated as a multiplicatively closed set (it is then possible to replace ${\mathcal A}$ by   the set of powers of an $\alpha$-special element) $R$ is a $G$-domain (see \cite[Theorem 19]{KaplanskyC}). Thus, by \cite[Theorem 156]{KaplanskyC}, $R$ is semilocal. The case when $R=K[[X]]$, the ring of formal power series and $\alpha$ is the  automorphism given by $\alpha(X)=qX$, for $q\in K$ not a root of unity, is a key example to consider.

A second point where necessary and sufficient conditions fail to match is in the affine case, where the proof in one direction of Theorem \ref{affine1} (= \ref{affine}) needs the base field to be uncountable. So we ask:

\begin{question}\label{qn2} Suppose that $R$ is an affine $k$-algebra and $S$ satisfies $(\diamond)$. Are all simple $S$-modules finite dimensional over $k$?
\end{question}

It is worth noting in connection with Question \ref{qn2} that $(\diamond)$ is in general not well-behaved with regard to change of base field. Consider, for example, any polycyclic-by-finite group $G$ which is not abelian-by-finite; for instance, the Heisenberg group
$$ G \quad = \quad \langle x,y,z : [x,y] = z, z \textrm{ central} \rangle. $$
Choose a prime $p$, and let $k$ be the field of $p$ elements. Then $kG$ satisfies $(\diamond)$, as already noted in $\S$\ref{history}, thanks to Roseblade's theorems \cite{Roseblade}, \cite{Roseblade1} (or by \cite{Hall} if we take $G$ to be nilpotent as above. Now let $F$ be any field of charactersitic $p$ which is \emph{not} algebraic over $k$. Then $FG = F \otimes_k kG$ does not satisfy $(\diamond)$ by \cite{Musson82}.

Haunting the results of $\S$6 and the examples of $\S$7 is the fundamental, if vague, question:

\begin{question}\label{qn3} Suppose that $R$ is $k$-affine and $\alpha \in \mathrm{Aut}_{k-\mathrm{alg}}(R)$. Are there ``reasonable'' necessary and sufficient  hypotheses on $R$ and $\alpha$ which determine when all the simple $S$-modules are finite-dimensional?
\end{question}

The examples stemming from group examples considered in $\S$\ref{groupring} show that the situation regarding Question \ref{qn3}  in positive characteristic is undoubtedly rather delicate. By contrast, in characteristic 0, there is, so far as we are aware, no example known at present where $S$ is $k$-affine, has all its simple modules finite dimensional, but does \emph{not} satisfy a polynomial identity. Note however that this includes specific examples studied in $\S\S$\ref{2poly} and \ref{Jordan} where the question of primitivity remains open. We therefore end by recording those concrete examples, which should be regarded as test cases for Question \ref{qn3}, and whose resolution seems to be potentially interesting from the perspectives of dynamical systems and of algebraic geometry.

The first of these two test cases repeats the question asked after  Theorem \ref{2dim}.

\begin{question}\label{qn4} Let $S = S_{\C,2,\alpha} = \C[x,y][\theta; \alpha]$, where $\alpha$ is a square automorphism. For example, $\alpha$ could be the map sending $x$ to $y$ and $y$ to $x + y^2$. Is $S$ primitive? Equivalently, does $S$ have an infinite dimensional simple module?
\end{question}

The final test case  seeks a resolution of the gap left in the proof of \cite[Proposition 7.13]{Jordan}.
\begin{question}\label{qn5} Let $S=\C[x,x^{-1},y,y^{-1}][\theta; \alpha]$, where $\alpha(y)=x$ and $\alpha(x)=yx^{-1}$. Is $S$ primitive? Equivalently, does $S$ have an infinite dimensional simple module?
\end{question}

\section*{Acknowledgements}
Some of this work was done while the second author visited the University of Glasgow and the University of Warsaw, supported by grant SFRH/BSAB/1286/2012. She would like to thank the two Universities for their hospitality.  The second author was also partially supported by CMUP (UID/MAT/00144/2013), which is funded by FCT (Portugal) with national (MEC) and European structural funds (FEDER), under the partnership agreement PT2020 and by the Warsaw Center of Mathematics and Computer Science.

The first and third authors would like to thank the University of Porto and its staff for hospitality and good working conditions. The third author acknowledges the support of CMUP and of Polish National Center of Science Grant No. DEC-2011/03/B/ST1/04893.

We are thankful to David Jordan, Christian Lomp, Ant\'onio Machiavelo and Carlos Rito for very helpful discussions regarding the behaviour of orbits in $\C^2$ under the action of  automorphisms.


\bigskip

  \textsc{School of Mathematics and Statistics, University of Glasgow, Glasgow G12 8SQ, Scotland}\par\nopagebreak
  \textit{E-mail address}: \texttt{Ken.Brown@glasgow.ac.uk}

  \medskip
   \textsc{Departamento de Matem\'atica, Faculdade de Ci\^encias, Universidade do Porto, 4169-007 Porto, Portugal}\par\nopagebreak
  \textit{E-mail address}: \texttt{pbcarval@fc.up.pt}

  \medskip
  
   \textsc{Institute of Mathematics, Warsaw University,  Banacha 2, 02-097 Warsaw, Poland}
  \par\nopagebreak
  \textit{E-mail address}:   \texttt{jmatczuk@mimuw.edu.pl}

  \medskip

\end{document}